\documentclass[12pt]{article}

\usepackage{amsmath,amssymb,amsthm,amscd,bm}

\usepackage{color}

\allowdisplaybreaks

\numberwithin{equation}{section}

\newcommand{\Fg}{\mathfrak{g}}
\newcommand{\Fh}{\mathfrak{h}}
\newcommand{\Fb}{\mathfrak{b}}
\newcommand{\FI}{\mathfrak{I}}

\newcommand{\BZ}{\mathbb{Z}}
\newcommand{\BQ}{\mathbb{Q}}
\newcommand{\BR}{\mathbb{R}}
\newcommand{\BC}{\mathbb{C}}

\newcommand{\BX}{\mathbb{X}}

\newcommand{\BU}{\mathbb{U}}

\newcommand{\CL}{\mathcal{L}}
\newcommand{\CB}{\mathcal{B}}

\newcommand{\sD}{\mathsf{D}}
\newcommand{\sT}{\mathsf{T}}

\newcommand{\ve}{\varepsilon}
\newcommand{\vp}{\varphi}
\newcommand{\vpi}{\varpi}
\newcommand{\eps}{\epsilon}

\newcommand{\ba}{\mathbf{a}}
\newcommand{\be}{\mathbf{e}}
\newcommand{\bp}{\mathbf{p}}
\newcommand{\bq}{\mathbf{q}}
\newcommand{\bx}{\mathbf{x}}
\newcommand{\bw}{\mathbf{w}}
\newcommand{\bzero}{\bm{0}}
\newcommand{\brho}{\bm{\rho}}
\newcommand{\bchi}{\bm{\chi}}

\newcommand{\q}{\mathsf{q}}
\newcommand{\lng}{w_{\circ}}

\newcommand{\rr}{\Delta_{\af}}
\newcommand{\prr}{\Delta_{\af}^{+}}

\newcommand{\mcr}[1]{\lfloor #1 \rfloor}
\newcommand{\xcr}[1]{\lceil #1 \rceil}
\newcommand{\edge}[1]{\xrightarrow{\,#1\,}}

\newcommand{\si}{\frac{\infty}{2}}
\newcommand{\sell}{\ell^{\si}}
\newcommand{\sil}{\prec}
\newcommand{\sile}{\preceq}
\newcommand{\sig}{\succ}
\newcommand{\sige}{\succeq}

\newcommand{\SLS}{\mathbb{B}^{\si}}
\newcommand{\QLS}{\mathrm{QLS}}
\newcommand{\EQB}{\mathrm{EQB}}

\newcommand{\cl}{\mathrm{cl}}
\newcommand{\af}{\mathrm{af}}

\newcommand{\wt}{\mathop{\rm wt}\nolimits}
\newcommand{\fin}{\mathop{\rm fin}\nolimits}
\newcommand{\fwt}{\mathop{\rm fin}\nolimits}
\newcommand{\qwt}{\mathop{\rm nul}\nolimits}

\newcommand{\gch}{\mathop{\rm gch}\nolimits}
\newcommand{\Hom}{\mathop{\rm Hom}\nolimits}
\newcommand{\Img}{\mathop{\rm Image}\nolimits}

\newcommand{\Par}{\mathop{\rm Par}\nolimits}
\newcommand{\Conn}{\mathop{\rm Conn}\nolimits}

\newcommand{\Kg}{\mathbb{K}}
\newcommand{\Kl}{\mathbf{K}}
\newcommand{\Kc}[1]{\SLS_{\sige #1;\not> #1}}

\newcommand{\pair}[2]{\langle #1,\,#2 \rangle}

\newcommand{\ol}[1]{\overline{#1}}
\newcommand{\ti}[1]{\widetilde{#1}}

\newcommand{\xw}{\xcr{w}}
\newcommand{\Iw}{I_{\xcr{w}}}
\newcommand{\Qvw}{Q_{\Iw}^{\vee,+}}

\newcommand{\SB}{\mathrm{BG}^{\si}(W_{\af})}
\newcommand{\QB}{\mathrm{QBG}(W)}

\newcommand{\J}{S}
\newcommand{\Jc}{I \setminus \J}
\newcommand{\QJ}{Q_{\J}}
\newcommand{\QJv}{Q_{\J}^{\vee}}
\newcommand{\QJvp}{Q_{\J}^{\vee,+}}
\newcommand{\DeJ}{\Delta_{\J}}
\newcommand{\PJ}{\Pi^{\J}}
\newcommand{\WJ}{W_{\J}}
\newcommand{\WJu}{W^{\J}}
\newcommand{\SBJ}{\mathrm{BG}^{\si}((\WJu)_{\af})}
\newcommand{\SBa}{\mathrm{BG}^{\si}_{a\lambda}((\WJu)_{\af})}
\newcommand{\SBb}[1]{\mathrm{BG}^{\si}_{#1\lambda}((\WJu)_{\af})}
\newcommand{\QBJ}{\mathrm{QBG}(\WJu)}
\newcommand{\QBa}{\mathrm{QBG}_{a\lambda}(\WJu)}
\newcommand{\QBb}[1]{\mathrm{QBG}_{#1\lambda}(\WJu)}
\newcommand{\kq}[1]{K^{\J}_{#1}}

\newcommand{\kp}[1]{K_{#1}}
\newcommand{\kw}{K_{\xcr{w}}}

\theoremstyle{plain}
\newtheorem{thm}{Theorem}[section]
\newtheorem{lem}[thm]{Lemma}
\newtheorem{prop}[thm]{Proposition}

\newtheorem{ithm}{Theorem}

\theoremstyle{definition}
\newtheorem{dfn}[thm]{Definition}

\theoremstyle{remark}
\newtheorem{rem}[thm]{Remark}

\newenvironment{enu}{%
 \begin{enumerate}%
}{\end{enumerate}}

\begin{document}

\setlength{\baselineskip}{18pt}

\title{\Large\bf 
Level-zero van der Kallen modules \\ and specialization of 
\\ nonsymmetric Macdonald polynomials at $t = \infty$%
}
\author{%
Satoshi Naito \\ 
 \small Department of Mathematics, Tokyo Institute of Technology, \\
 \small 2-12-1 Oh-okayama, Meguro-ku, Tokyo 152-8551, Japan \\
 \small (e-mail: {\tt naito@math.titech.ac.jp}) \\[3mm]
and \\[3mm]
Daisuke Sagaki \\ 
 \small Institute of Mathematics, University of Tsukuba, \\
 \small 1-1-1 Tennodai, Tsukuba, Ibaraki 305-8571, Japan \\
 \small (e-mail: {\tt sagaki@math.tsukuba.ac.jp})
}
\date{}
\maketitle

%
\begin{abstract} \setlength{\baselineskip}{13pt}

Let $\lambda \in P^{+}$ be a level-zero dominant integral weight, 
and $w$ an arbitrary coset representative of minimal length for the cosets 
in $W/W_{\lambda}$, where $W_{\lambda}$ is the stabilizer of $\lambda$ in a finite Weyl group $W$.
In this paper, we give a module $\Kg_{w}^{-}(\lambda)$ over 
the negative part of a quantum affine algebra whose graded character is 
identical to the specialization at $t = \infty$ of 
the nonsymmetric Macdonald polynomial $E_{w \lambda}(q,\,t)$ 
multiplied by a certain explicit finite product of 
rational functions of $q$ of the form $(1 - q^{-r})^{-1}$ for a positive integer $r$.
This module $\Kg_{w}^{-}(\lambda)$ (called a level-zero van der Kallen module) is 
defined to be the quotient module of the level-zero Demazure module $V_{w}^{-}(\lambda)$ 
by the sum of the submodules $V_{z}^{-}(\lambda)$ for all those coset representatives $z$ 
of minimal length for the cosets in $W/W_{\lambda}$ such that $z > w$ in the Bruhat order $<$ on $W$.
\end{abstract}
%
%
\section{Introduction.} 
\label{sec:intro}

In our previous paper \cite{NS16}, 
we computed the graded character $\gch V_{e}^{-}(\lambda)$ of 
the Demazure submodule $V_{e}^{-}(\lambda)$ of 
a level-zero extremal weight module $V(\lambda)$ 
over the quantum affine algebra $U_{\q}(\Fg_{\af})$ 
associated to a nontwisted affine Lie algebra $\Fg_{\af}$, 
where $\lambda \in P^{+}$ is a level-zero dominant integral weight and 
$e$ is the identity element of the affine Weyl group $W_{\af}$.
The main result of \cite{NS16} states that the graded character $\gch  V_{e}^{-}(\lambda)$ is 
identical to the specialization $E_{\lng \lambda}(q, 0)$ at $t = 0$ of 
the nonsymmetric Macdonald polynomial multiplied by the inverse of 
the finite product $\prod_{i \in I} \prod_{r = 1}^{ \pair{\lambda}{\alpha^{\vee}_{i}} } (1 - q^{-r})$, 
where $\lng$ is the longest element of the finite Weyl group $W \subset W_{\af}$ and 
$q$ denotes the formal exponential $\be^{\delta}$, 
with $\delta$ the null root of $\Fg_{\af}$.
Also, in \cite{NNS1}, we computed the graded character $\gch  V_{\lng}^{-}(\lambda)$ of 
the Demazure submodule $V_{\lng}^{-}(\lambda)$ of $V(\lambda)$, 
and proved that it is identical to the specialization $E_{\lng \lambda}(q, \infty)$ at $t = \infty$ 
of the nonsymmetric Macdonald polynomial multiplied by the inverse of the same finite product as above.
Moreover, in \cite{NNS1}, for an arbitrary element $w$ of the finite Weyl group $W$, 
we obtained combinatorial formulas for the graded character $\gch V_{w}^{-}(\lambda)$ of 
the Demazure submodule $V_{w}^{-}(\lambda)$ of $V(\lambda)$ and 
the specialization $E_{w \lambda}(q, \infty)$ at $t = \infty$ of 
the nonsymmetric Macdonald polynomial, described in terms of 
quantum Lakshmibai-Seshadri paths introduced in \cite{LNSSS16}.

Independently, Feigin-Makedonskyi \cite{FM17} introduced 
a family of finite-dimensional modules (called generalized Weyl modules),
indexed by the elements $w$ of the finite Weyl group $W$,
over the Iwahori subalgebra $\FI:=\Fb \oplus (z \BC[z] \otimes \Fg)$ 
of the current algebra $\Fg[z] := \BC[z] \otimes \Fg$ 
associated to the finite-dimensional simple Lie algebra $\Fg \subset \Fg_{\af}$
with Borel subalgebra $\Fb$, 
and proved that for the cases $w = e$ and $w = \lng$, their graded characters are 
identical to the specializations at $t = \infty$ and $t = 0$ of 
the nonsymmetric Macdonald polynomial $E_{\lng \lambda}(q, t)$, respectively.
Here we mention that the graded character of a generalized Weyl module 
indexed by a general element $w \ne e,\,\lng$ of $W$
is not identical to the specialization at $t = 0$ nor $t = \infty$ 
of a nonsymmetric Macdonald polynomial.

Afterward, in \cite{Kat16}, for an arbitrary element $w$ of 
the finite Weyl group $W$, Kato gave an algebro-geometric construction of 
a finite-dimensional $\FI$-module whose graded character is 
identical to the specialization $E_{w \lambda}(q, \infty)$ at $t = \infty$ 
of the nonsymmetric Macdonald polynomial. Based on this result,
Feigin-Kato-Makedonskyi \cite{FKM} gave an algebraic description of 
these $\FI$-modules (denoted by $\BU_{\sigma(\lambda_{-})}$, 
with $\lambda_{-}$ antidominant and $\sigma \in W$) 
by generators and relations, 
which are similar to that of generalized Weyl modules given in \cite{FM17}.

The purpose of this paper is to give a module 
(which we call a level-zero van der Kallen module) 
over the negative part $U_{\q}^{-}(\Fg_{\af})$ of $U_{\q}(\Fg_{\af})$ 
whose graded character is identical to 
the specialization $E_{w \lambda}(q, \infty)$ at $t = \infty$ of 
the nonsymmetric Macdonald polynomial multiplied by 
the inverse of the same product as above for an arbitrary $w \in W$; 
note that the Demazure submodule $V_{w}^{-}(\lambda)$ is a $U_{\q}^{-}(\Fg_{\af})$-submodule 
of the level-zero extremal weight module $V(\lambda)$.
Also, we give an explicit description of the crystal basis of 
the level-zero van der Kallen module in terms of 
semi-infinite Lakshmibai-Seshadri paths (or, quantum Lakshmibai-Seshadri paths).

Let us explain our main result more precisely. 
Let $\lambda \in P^{+}$ be a level-zero dominant integral weight. 
We set $\J = \J_{\lambda} := 
\bigl\{i \in I \mid \pair{\lambda}{\alpha^{\vee}_{i}} = 0 \bigr\}$, 
and let $\WJu$ denote the set of minimal-length coset representatives for the cosets in $W/\WJ$,
where $W_{S} := \langle s_{i} \mid i \in S \rangle$ is the subgroup of 
the finite Weyl group $W = \langle s_{i} \mid i \in I \rangle$; 
we denote by $\lng(S)$ the longest element of $W_{S}$.
For $w \in \WJu$, we set
\begin{equation*}
\Kg_{w}^{-}(\lambda):=
 V_{w}^{-}(\lambda)\Biggm/
 \sum_{ z \in \WJu, \, z > w } 
 V_{z}^{-}(\lambda); 
\end{equation*}
here we know from \cite{NS16} that $V_{z}^{-}(\lambda) \subset V_{w}^{-}(\lambda)$ 
for all $z \in \WJu$ such that $z > w$ in the Bruhat order $<$ 
on the finite Weyl group $W$. We call 
the module $\Kg_{w}^{-}(\lambda)$ a level-zero van der Kallen module.
Our main result is the following.
\begin{ithm} \label{ithm}
Let $\lambda \in P^{+}$ be a level-zero dominant integral weight, 
and let $w \in \WJu$. Then the graded character $\gch \Kg_{w}$ can be expressed as follows{\rm :}
\begin{equation*}
\gch \Kg_{w}^{-}(\lambda) = 
 \left(\prod_{i \in I} 
  \prod_{r = 1}^{\pair{\lambda}{\alpha^{\vee}_{i}} - \eps_{i}} (1 - q^{-r})\right)^{-1} 
  E_{w \lambda}(q, \infty),
\end{equation*}
where
\begin{equation*}
\eps_{i} = \eps_{i}(\xw) :=
\begin{cases}
1 & \text{\rm if } \quad \xw s_i > \xw, \\
0 & \text{\rm if } \quad \xw s_i < \xw.
\end{cases}
\end{equation*}
Here, $\xw := w \lng(S) \in W$ denotes the maximal-length representative for the coset $w W_{S} \in W/W_{S}$.
\end{ithm}

Our proof of this theorem is a crystal-theoretic one, 
which is based on the formulas for the graded character $\gch V_{w}^{-}(\lambda)$ 
and the specialization $E_{w \lambda}(q, \infty)$ at $t = \infty$ of 
the nonsymmetric Macdonald polynomial in terms of quantum Lakshmibai-Seshadri paths 
obtained in \cite{NNS1}. In fact, the crystal basis of the level-zero 
van der Kallen module $\Kg_{w}^{-}(\lambda)$ can be realized 
as the set of those semi-infinite Lakshmibai-Seshadri paths $\pi$ of shape $\lambda$ 
whose final direction $\kappa(\pi)$ lies in the set
\begin{equation*}
\kq{w} := ((\WJu)_{\af})_{\sige w} \setminus 
 \bigcup_{z \in \WJu, \, z > w} ( (\WJu)_{\af} )_{\sige z},
\end{equation*}
where $( (\WJu)_{\af} )_{\sige x} := \bigl\{ y \in (\WJu)_{\af} \mid y \sige x\}$ 
for $x \in (\WJu)_{\af}$, with $(\WJu)_{\af}$ 
the set of Peterson's coset representatives 
for the cosets in $W_{\af}/(\WJ)_{\af}$ 
(for the notation, see Section~\ref{sec:main}).
Also, in the course of the proof of Theorem~\ref{ithm}, 
we obtain a non-recursive description of the (strange) subsets 
$\EQB(w)$ for $w \in W$, introduced in \cite{NNS1} and studied in \cite{NNS2}, 
in terms of the semi-infinite Bruhat graph (or, quantum Bruhat graph); 
see Section~\ref{sec:kappa} for details.
Here we should mention that in contrast to the arguments in \cite{FKM}, 
our proof works uniformly for all nontwisted affine Lie algebras and does not use 
the reduction to the rank two cases, though the construction of 
the module $\Kg_{w}^{-}(\lambda)$ above is inspired by \cite{FKM}.

In Appendix~\ref{sec:prf2}, we give another proof of Theorem~\ref{ithm}, 
which is based on the recursive formula for the specialization at $t = \infty$ of 
the nonsymmetric Macdonald polynomials due to Cherednik-Orr (\cite{CO});
we would like to thank a referee for suggesting this recursive proof.
A key to this proof is Lemma~\ref{lem:dj2} (or rather equivalently, 
Proposition~\ref{prop:interval}). 
As a consequence of this lemma, we obtain a formula, 
which expresses the graded character of a Demazure submodule of 
a level-zero extremal weight module as an explicit sum of 
the graded characters of level-zero van der Kallen modules:
for each $w \in \WJu$,
\begin{equation*}
\gch V_{w}^{-}(\lambda) = \sum_{v \in \WJu, \, v \ge w} \gch \Kg_{v}^{-}(\lambda).
\end{equation*}

Finally, in Appendix~\ref{sec:apdx}, 
we remark that the cyclic vector of 
(the classical limit , i.e., the limit at $\q \to 1$ of) 
the level-zero van der Kallen module $\Kg_{w}^{-}(\lambda)$ 
for $w \in \WJu$ satisfies the same relations (see Lemma~\ref{lem:rel}) 
as those for the negative-roots version of the $\FI$-modules 
$\BU_{\sigma(\lambda_{-})}$ for $\sigma \in W$ introduced in \cite{FKM}.
Because both of these modules have the same graded character (at least in simply-laced cases), 
they are in fact isomorphic if the notational convention is suitably adjusted. 
Hence level-zero van der Kallen modules $\Kg_{w}^{-}(\lambda)$ 
can be thought of as a quantum analog of 
the $\FI$-modules $\BU_{\sigma(\lambda_{-})}$ in \cite{FKM}.

This paper is organized as follows. 
In Section~\ref{sec:main}, we fix our notation 
for affine Lie algebras, and recall some basic facts 
about the (parabolic) semi-infinite Bruhat graph. 
Next, we briefly review fundamental results on level-zero extremal weight modules 
and their Demazure submodules. Also, we define 
level-zero van der Kallen modules and state our main result above.
In Section~\ref{sec:SLS}, we review the realization of the crystal bases of 
level-zero extremal weight modules by semi-infinite Lakshmibai-Seshadri paths.
In Section~\ref{sec:kappa}, 
we first recall some basic facts about the (parabolic) quantum Bruhat graph, 
and then review from \cite{NNS1} a recursive description of 
the subsets $\EQB(w) \subset W$, $w \in W$, which are needed in the formula 
for the specialization $E_{w \lambda}(q, \infty)$. Also, we obtain a condition 
for the final directions of semi-infinite Lakshmibai-Seshadri paths 
forming the crystal basis of a level-zero van der Kallen module.
In Section~\ref{sec:prf-main}, by using this condition, 
we give a proof of our main result above.
In Appendix~\ref{sec:prf2}, we give a recursive proof of our main result.
In Appendix~\ref{sec:apdx}, we mention some relations satisfied 
by the cyclic vector of the level-zero van der Kallen module.
%
%
\subsection*{Acknowledgments.}
S.N. was partially supported by 
JSPS Grant-in-Aid for Scientific Research (B) 16H03920. 
D.S. was partially supported by 
JSPS Grant-in-Aid for Scientific Research (C) 15K04803. 

%
\section{Main result.}
\label{sec:main}
%
%
\subsection{Affine Lie algebras.}
\label{subsec:liealg}

Let $\Fg$ be a finite-dimensional simple Lie algebra over $\BC$
with Cartan subalgebra $\Fh$. 
Denote by $\{ \alpha_{i}^{\vee} \}_{i \in I}$ and 
$\{ \alpha_{i} \}_{i \in I}$ the set of simple coroots and 
simple roots of $\Fg$, respectively, and set
$Q := \bigoplus_{i \in I} \BZ \alpha_i$, 
$Q^{+} := \sum_{i \in I} \BZ_{\ge 0} \alpha_i$, and 
$Q^{\vee} := \bigoplus_{i \in I} \BZ \alpha_i^{\vee}$, 
$Q^{\vee,+} := \sum_{i \in I} \BZ_{\ge 0} \alpha_i^{\vee}$; 
for $\xi,\,\zeta \in Q^{\vee}$, we write $\xi \ge \zeta$ if $\xi-\zeta \in Q^{\vee,+}$. 
Let $\Delta$, $\Delta^{+}$, and $\Delta^{-}$ be 
the set of roots, positive roots, and negative roots of $\Fg$, respectively, 
with $\theta \in \Delta^{+}$ the highest root of $\Fg$. 
For a root $\alpha \in \Delta$, we denote by $\alpha^{\vee}$ its dual root. 
We set $\rho:=(1/2) \sum_{\alpha \in \Delta^{+}} \alpha$. 
Also, let $\vpi_{i}$, $i \in I$, denote the fundamental weights for $\Fg$, and set
%
%
\begin{equation} \label{eq:P-fin}
P:=\bigoplus_{i \in I} \BZ \vpi_{i}, \qquad 
P^{+} := \sum_{i \in I} \BZ_{\ge 0} \vpi_{i}. 
\end{equation} 

Let $\Fg_{\af} = \bigl(\BC[z,z^{-1}] \otimes \Fg\bigr) \oplus \BC c \oplus \BC d$ be 
the untwisted affine Lie algebra over $\BC$ associated to $\Fg$, 
where $c$ is the canonical central element, and $d$ is 
the scaling element (or the degree operator), 
with Cartan subalgebra $\Fh_{\af} = \Fh \oplus \BC c \oplus \BC d$. 
We regard an element $\mu \in \Fh^{\ast}:=\Hom_{\BC}(\Fh,\,\BC)$ as an element of 
$\Fh_{\af}^{\ast}$ by setting $\pair{\mu}{c}=\pair{\mu}{d}:=0$, where 
$\pair{\cdot\,}{\cdot}:\Fh_{\af}^{\ast} \times \Fh_{\af} \rightarrow \BC$ denotes
the canonical pairing of $\Fh_{\af}^{\ast}:=\Hom_{\BC}(\Fh_{\af},\,\BC)$ and $\Fh_{\af}$. 
Let $\{ \alpha_{i}^{\vee} \}_{i \in I_{\af}} \subset \Fh_{\af}$ and 
$\{ \alpha_{i} \}_{i \in I_{\af}} \subset \Fh_{\af}^{\ast}$ be the set of 
simple coroots and simple roots of $\Fg_{\af}$, respectively, 
where $I_{\af}:=I \sqcup \{0\}$; note that 
$\pair{\alpha_{i}}{c}=0$ and $\pair{\alpha_{i}}{d}=\delta_{i0}$ 
for $i \in I_{\af}$. 
Denote by $\delta \in \Fh_{\af}^{\ast}$ the null root of $\Fg_{\af}$; 
recall that $\alpha_{0}=\delta-\theta$. 
Also, let $\Lambda_{i} \in \Fh_{\af}^{\ast}$, $i \in I_{\af}$, 
denote the fundamental weights for $\Fg_{\af}$ such that $\pair{\Lambda_{i}}{d}=0$, 
and set 
%
%
\begin{equation} \label{eq:P}
P_{\af} := 
  \left(\bigoplus_{i \in I_{\af}} \BZ \Lambda_{i}\right) \oplus 
   \BZ \delta \subset \Fh^{\ast}, \qquad 
P_{\af}^{0}:=\bigl\{\mu \in P_{\af} \mid \pair{\mu}{c}=0\bigr\};
\end{equation}
notice that $P_{\af}^{0}=P \oplus \BZ \delta$, and that
$\pair{\mu}{\alpha_{0}^{\vee}} = - \pair{\mu}{\theta^{\vee}}$ 
for $\mu \in P_{\af}^{0}$. We remark that for each $i \in I$, 
$\vpi_{i}$ is equal to $\Lambda_{i}-\pair{\Lambda_{i}}{c}\Lambda_{0}$, 
which is called the level-zero fundamental weight in \cite{Kas02}. 

Let $W := \langle s_{i} \mid i \in I \rangle$ and 
$W_{\af} := \langle s_{i} \mid i \in I_{\af} \rangle$ be the (finite) Weyl group of $\Fg$ and 
the (affine) Weyl group of $\Fg_{\af}$, respectively, 
where $s_{i}$ is the simple reflection with respect to $\alpha_{i}$ 
for each $i \in I_{\af}$. We denote by $\ell:W_{\af} \rightarrow \BZ_{\ge 0}$ 
the length function on $W_{\af}$, whose restriction to $W$ agrees with 
the one on $W$, by $e \in W \subset W_{\af}$ the identity element, 
and by $\lng \in W$ the longest element. 
Denote by $\ge$ the (ordinary) Bruhat order on $W$. 
For each $\xi \in Q^{\vee}$, let $t_{\xi} \in W_{\af}$ denote 
the translation in $\Fh_{\af}^{\ast}$ by $\xi$ (see \cite[Sect.~6.5]{Kac}); 
for $\xi \in Q^{\vee}$, we have 
%
%
\begin{equation}\label{eq:wtmu}
t_{\xi} \mu = \mu - \pair{\mu}{\xi}\delta \quad 
\text{if $\mu \in \Fh_{\af}^{\ast}$ satisfies $\pair{\mu}{c}=0$}.
\end{equation}
Then, $\bigl\{ t_{\xi} \mid \xi \in Q^{\vee} \bigr\}$ forms 
an abelian normal subgroup of $W_{\af}$, in which $t_{\xi} t_{\zeta} = t_{\xi + \zeta}$ 
holds for $\xi,\,\zeta \in Q^{\vee}$. Moreover, we know from \cite[Proposition 6.5]{Kac} that
\begin{equation*}
W_{\af} \cong 
 W \ltimes \bigl\{ t_{\xi} \mid \xi \in Q^{\vee} \bigr\} \cong W \ltimes Q^{\vee}. 
\end{equation*}

Denote by $\rr$ the set of real roots of $\Fg_{\af}$, and 
by $\prr \subset \rr$ the set of positive real roots; 
we know from \cite[Proposition 6.3]{Kac} that
$\rr = 
\bigl\{ \alpha + n \delta \mid \alpha \in \Delta,\, n \in \BZ \bigr\}$, 
and 
$\prr = 
\Delta^{+} \sqcup 
\bigl\{ \alpha + n \delta \mid \alpha \in \Delta,\, n \in \BZ_{> 0}\bigr\}$. 
For $\beta \in \rr$, we denote by $\beta^{\vee} \in \Fh_{\af}$ 
its dual root, and $s_{\beta} \in W_{\af}$ the corresponding reflection; 
if $\beta \in \rr$ is of the form $\beta = \alpha + n \delta$ 
with $\alpha \in \Delta$ and $n \in \BZ$, then 
$s_{\beta} =s_{\alpha} t_{n\alpha^{\vee}} \in W \ltimes Q^{\vee}$.

Finally, let $U_{\q}(\Fg_{\af})$ (resp., $U_{\q}'(\Fg_{\af})$)
denote the quantized universal enveloping algebra over $\BC(\q)$ 
associated to $\Fg_{\af}$ (resp., $[\Fg_{\af},\Fg_{\af}]$), 
with $E_{i}$ and $F_{i}$, $i \in I_{\af}$, the Chevalley generators 
corresponding to $\alpha_{i}$ and $-\alpha_{i}$, respectively.  
We denote by $U_{\q}^{-}(\Fg_{\af})$ 
the negative part of $U_{\q}(\Fg_{\af})$, that is, 
the $\BC(\q)$-subalgebra of $U_{\q}(\Fg_{\af})$ generated by $F_{i}$, $i \in I_{\af}$. 
%
%
\subsection{Parabolic semi-infinite Bruhat graph.}
\label{subsec:SiBG}

In this subsection, we take and fix an arbitrary subset $\J \subset I$. We set 
$\QJ := \bigoplus_{i \in \J} \BZ \alpha_i$, 
$\QJv := \bigoplus_{i \in \J} \BZ \alpha_i^{\vee}$,  
$\QJvp := \sum_{i \in \J} \BZ_{\ge 0} \alpha_i^{\vee}$, 
$\DeJ := \Delta \cap \QJ$, 
$\DeJ^{\pm} := \Delta^{\pm} \cap \QJ$, 
$\WJ := \langle s_{i} \mid i \in \J \rangle$, and 
$\rho_{\J}:=(1/2) \sum_{\alpha \in \DeJ^{+}} \alpha$; 
we denote by
$[\,\cdot\,]^{\J} : Q^{\vee} \twoheadrightarrow Q_{\Jc}^{\vee}$
the projection from $Q^{\vee}=Q_{\Jc}^{\vee} \oplus \QJv$
onto $Q_{\Jc}^{\vee}$ with kernel $\QJv$. 
Let $\WJu$ denote the set of minimal-length coset representatives 
for the cosets in $W/\WJ$; we know from \cite[Sect.~2.4]{BB} that 
%
%
\begin{equation} \label{eq:mcr}
\WJu = \bigl\{ w \in W \mid 
\text{$w \alpha \in \Delta^{+}$ for all $\alpha \in \DeJ^{+}$}\bigr\}.
\end{equation}
For $w \in W$, we denote by $\mcr{w}=\mcr{w}^{\J} \in \WJu$ 
(resp., $\xw=\xw^{\J}$) the minimal-length (resp., maximal-length) coset 
representative for the coset $w \WJ$ in $W/\WJ$; 
note that $\xw = \mcr{w}\lng(\J)$, where 
$\lng(\J)$ denotes the longest element of the subgroup $\WJ$ of $W$. 
Also, following \cite{Pet97} 
(see also \cite[Sect.~10]{LS10}), we set
\begin{align}
(\DeJ)_{\af} 
  & := \bigl\{ \alpha + n \delta \mid 
  \alpha \in \DeJ,\,n \in \BZ \bigr\} \subset \Delta_{\af}, \\
(\DeJ)_{\af}^{+}
  &:= (\DeJ)_{\af} \cap \prr = 
  \DeJ^{+} \sqcup \bigl\{ \alpha + n \delta \mid 
  \alpha \in \DeJ,\, n \in \BZ_{> 0} \bigr\}, \\
\label{eq:stabilizer}
(\WJ)_{\af} 
 & := \WJ \ltimes \bigl\{ t_{\xi} \mid \xi \in \QJv \bigr\}
   = \bigl\langle s_{\beta} \mid \beta \in (\DeJ)_{\af}^{+} \bigr\rangle, \\
\label{eq:Pet}
(\WJu)_{\af}
 &:= \bigl\{ x \in W_{\af} \mid 
 \text{$x\beta \in \prr$ for all $\beta \in (\DeJ)_{\af}^{+}$} \bigr\};
\end{align}
if $\J = \emptyset$, then 
$(W^{\emptyset})_{\af}=W_{\af}$ and $(W_{\emptyset})_{\af}=\bigl\{e\bigr\}$. 
We know from \cite{Pet97} (see also \cite[Lemma~10.6]{LS10}) that 
for each $x \in W_{\af}$, there exist a unique 
$x_1 \in (\WJu)_{\af}$ and a unique $x_2 \in (\WJ)_{\af}$ 
such that $x = x_1 x_2$; let 
%
%
\begin{equation} \label{eq:PiJ}
\PJ : W_{\af} \twoheadrightarrow (\WJu)_{\af}, \quad x \mapsto x_{1}, 
\end{equation}
denote the projection, 
where $x= x_1 x_2$ with $x_1 \in (\WJu)_{\af}$ and $x_2 \in (\WJ)_{\af}$. 
%
%
\begin{lem} \label{lem:PiJ}
\mbox{}
\begin{enu}
\item It holds that 
%
%
\begin{equation} \label{eq:PiJ2}
\begin{cases}
\PJ(w)=\mcr{w} 
  & \text{\rm for all $w \in W$}; \\[1mm]
\PJ(xt_{\xi})=\PJ(x)\PJ(t_{\xi}) 
  & \text{\rm for all $x \in W_{\af}$ and $\xi \in Q^{\vee}$};
\end{cases}
\end{equation}
in particular, 
$(\WJu)_{\af} 
  = \bigl\{ w \PJ(t_{\xi}) \mid w \in \WJu,\,\xi \in Q^{\vee} \bigr\}$.

\item For each $\xi \in Q^{\vee}$, 
the element $\PJ(t_{\xi}) \in (\WJu)_{\af}$ is 
of the form{\rm:} $\PJ(t_{\xi})=ut_{\xi+\xi_{1}}$ 
for some $u \in \WJ$ and $\xi_{1} \in \QJv$. 

\item For $\xi,\,\zeta \in Q^{\vee}$, 
$\PJ(t_{\xi}) = \PJ(t_{\zeta})$ if and only if $\xi-\zeta \in \QJv$.

\end{enu}
\end{lem}
\begin{proof}
Part (1) follows from \cite[Proposition~10.10]{LS10}, and 
part (2) follows from \cite[(3.7)]{LNSSS}. 
The ``if'' part of part (3) is obvious by part (1) and 
the fact that $t_{\xi-\zeta} \in (\WJ)_{\af}$. 
The ``only if'' part of part (3) is obvious by part (2). 
\end{proof}
%
%
\begin{dfn} \label{dfn:sell}
Let $x \in W_{\af}$, and 
write it as $x = w t_{\xi}$ with $w \in W$ and $\xi \in Q^{\vee}$. 
We define the semi-infinite length $\sell(x)$ of $x$ by:
$\sell (x) = \ell (w) + 2 \pair{\rho}{\xi}$. 
\end{dfn}
%
%
\begin{dfn}[\cite{Lu80}, \cite{Lu97}; see also \cite{Pet97}] \label{dfn:SiB}
\mbox{}
\begin{enu}
\item The (parabolic) semi-infinite Bruhat graph $\SBJ$ 
is the $\prr$-labeled directed graph whose 
vertices are the elements of $(\WJu)_{\af}$, and 
whose directed edges are of the form: 
$x \edge{\beta} y$ for $x,y \in (\WJu)_{\af}$ and $\beta \in \prr$ 
such that $y=s_{\beta}x$ and $\sell (y) = \sell (x) + 1$. 
When $\J=\emptyset$, we write $\SB$ for 
$\mathrm{BG}^{\si}\bigl((W^{\emptyset})_{\af}\bigr)$. 

\item 
The (parabolic) semi-infinite Bruhat order is a partial order 
$\sile$ on $(\WJu)_{\af}$ defined as follows: 
for $x,\,y \in (\WJu)_{\af}$, we write $x \sile y$ 
if there exists a directed path in $\SBJ$ from $x$ to $y$; 
we write $x \sil y$ if $x \sile y$ and $x \ne y$. 
\end{enu}
\end{dfn}

\begin{rem}
In the case $\J = \emptyset$, the semi-infinite Bruhat order on $W_{\af}$ is
essentially the same as the generic Bruhat order introduced in \cite{Lu80}; 
see \cite[Appendix~A.3]{INS} for details. Also, for a general $\J$, 
the parabolic semi-infinite Bruhat order on $(\WJu)_{\af}$
is essentially the same as the partial order on $\J$-alcoves introduced in
\cite{Lu97} when we take a special point to be the origin.
\end{rem}
%
%
\begin{rem} \label{rem:SB}
It follows from the definition that 
the restriction of the semi-infinite Bruhat order $\sige$ on $(\WJu)_{\af}$ 
to $\WJu \subset (\WJu)_{\af}$ agrees with 
the (ordinary) Bruhat order $\ge$ on $\WJu$. 
\end{rem}
In Section~\ref{subsec:lemma} below, we recall some of 
the basic properties of the semi-infinite Bruhat order. 
%
%
\subsection{Crystal bases of extremal weight modules.}
\label{subsec:extremal}

In this subsection, we fix $\lambda \in P^{+} \subset P_{\af}^{0}$ 
(see \eqref{eq:P-fin} and \eqref{eq:P}). 
Let $V(\lambda)$ denote the extremal weight module of 
extremal weight $\lambda$ over $U_{\q}(\Fg_{\af})$, 
which is an integrable $U_{\q}(\Fg_{\af})$-module generated by 
a single element $v_{\lambda}$ with 
the defining relation that $v_{\lambda}$ is 
an extremal weight vector of weight $\lambda$; 
recall from \cite[Sect.~3.1]{Kas02} and \cite[Sect.~2.6]{Kas05} that 
$v_{\lambda}$ is an extremal weight vector of weight $\lambda$ 
if ($v_{\lambda}$ is a weight vector of weight $\lambda$ and) 
there exists a family $\{ v_{x} \}_{x \in W_{\af}}$ 
of weight vectors in $V(\lambda)$ such that $v_{e}=v_{\lambda}$, 
and such that for every $i \in I_{\af}$ and $x \in W_{\af}$ with 
$n:=\pair{x\lambda}{\alpha_{i}^{\vee}} \ge 0$ (resp., $\le 0$),
the equalities $E_{i}v_{x}=0$ and 
$F_{i}^{(n)}v_{x}=v_{s_{i}x}$ 
(resp., $F_{i}v_{x}=0$ and 
$E_{i}^{(-n)}v_{x}=v_{s_{i}x}$) hold, 
where for $i \in I_{\af}$ and $k \in \BZ_{\ge 0}$, 
the $E_{i}^{(k)}$ and $F_{i}^{(k)}$ are the $k$-th divided powers of 
$E_{i}$ and $F_{i}$, respectively;
note that $v_{x}$ is an extremal weight vector of weight $x\lambda$. 
We know from \cite[Proposition~8.2.2]{Kas94} that $V(\lambda)$ has 
a crystal basis $(\CL(\lambda),\CB(\lambda))$ and the corresponding 
global basis $\bigl\{G(b) \mid b \in \CB(\lambda)\bigr\}$; 
we denote by $u_{\lambda}$ the element of $\CB(\lambda)$ 
such that $G(u_{\lambda})=v_{\lambda}$. 
It follows from \cite[Sect.~7]{Kas94} that 
the affine Weyl group $W_{\af}$ acts on $\CB(\lambda)$ by
%
%
\begin{equation} \label{eq:W1}
s_{i} \cdot b:=
\begin{cases}
f_{i}^{n}b & \text{if $n:=\pair{\wt b}{\alpha_{i}^{\vee}} \ge 0$}, \\[1.5mm]
e_{i}^{-n}b & \text{if $n:=\pair{\wt b}{\alpha_{i}^{\vee}} \le 0$}, 
\end{cases}
\end{equation}
for $b \in \CB(\lambda)$ and $i \in I_{\af}$. 

We know the following from \cite{Kas02} (see also \cite[Sect.~5.2]{NS16}). 
Let $i \in I$. 
\begin{enu}
\item[(i)] 
The crystal graph of $\CB(\vpi_{i})$ is connected, and 
$\CB(\vpi_{i})_{\vpi_{i}+k\delta}=
\bigl\{ u_{\vpi_{i}+k\delta} \bigr\}$ for all $k \in \BZ$, 
where $u_{\vpi_{i}+k\delta}:=
t_{-k\alpha_{i}^{\vee}} \cdot u_{\vpi_{i}}$ for $k \in \BZ$. 
Therefore, $\dim V(\vpi_{i})_{\vpi_{i}+k\delta}=1$ for all $k \in \BZ$. 

\item[(ii)]
There exists a $U_{\q}'(\Fg_{\af})$-module 
automorphism $z_{i}:V(\vpi_i) \rightarrow V(\vpi_i)$ 
that maps $v_{\vpi_{i}}$ to $v_{\vpi_{i}+\delta}:=
G(u_{\vpi_{i}+\delta})$; thus, this map commutes 
with the Kashiwara operators $e_{j}$, $f_{j}$, 
$j \in I_{\af}$, on $V(\vpi_{i})$. 

\item[(iii)] We have $z_{i}(\CL(\vpi_{i})) \subset \CL(\vpi_{i})$. 
Hence the map $z_{i}:V(\vpi_{i}) \rightarrow V(\vpi_{i})$ 
induces a $\BC$-linear automorphism $z_{i}:
\CL(\vpi_{i})/\q\CL(\vpi_{i}) \rightarrow \CL(\vpi_{i})/\q\CL(\vpi_{i})$; 
this induced map commutes with the Kashiwara operators 
$e_{j}$, $f_{j}$, $j \in I_{\af}$, on $\CL(\vpi_{i})/\q\CL(\vpi_{i})$, and 
satisfies $z_{i}(u_{\vpi_{i}}) = u_{\vpi_{i}+\delta}$. 
Therefore, the map $z_{i}$ preserves $\CB(\vpi_{i})$. 
\end{enu}
Let us write $\lambda \in P^{+}$ as
$\lambda = \sum_{i \in I} m_{i} \vpi_{i}$, 
with $m_{i} \in \BZ_{\ge 0}$ for $i \in I$. 
We fix an arbitrary total order on $I$, and then set 
$\ti{V}(\lambda):=
\bigotimes_{i \in I} V(\vpi_{i})^{\otimes m_{i}}$.
By \cite[Eq.\,(4.8) and Corollary~4.15]{BN}, 
there exists an injective $U_{\q}(\Fg_{\af})$-module homomorphism
$\Phi_{\lambda}:V(\lambda) \hookrightarrow \ti{V}(\lambda)$
that maps $v_{\lambda}$ to $\ti{v}_{\lambda}:=
\bigotimes_{i \in I} v_{\vpi_{i}}^{\otimes m_{i}}$. 
For each $i \in I$ and $1 \le k \le m_{i}$, we define $z_{i,k}$ to be 
the $U_{\q}'(\Fg_{\af})$-module automorphism of $\ti{V}(\lambda)$ 
which acts as $z_{i}$ only on the $k$-th factor of $V(\vpi_{i})^{\otimes m_{i}}$ 
in $\ti{V}(\lambda)$, and as the identity map on the other factors of $\ti{V}(\lambda)$; 
these $z_{i,k}$'s, $i \in I$, $1 \le k \le m_{i}$, 
commute with each other. 
We define
%
%
\begin{equation} \label{eq:olpar}
\ol{\Par(\lambda)} := 
 \bigl\{ \brho = (\rho^{(i)})_{i \in I} \mid 
 \text{$\rho^{(i)}$ is a partition of length $\le m_{i}$ 
 for each $i \in I$} \bigr\}. 
\end{equation}
For $\brho=(\rho^{(i)})_{i \in I} \in \ol{\Par(\lambda)}$, we set
%
%
\begin{equation} \label{eq:Schur}
s_{\brho}(z^{-1}):=\prod_{i \in I} 
s_{\rho^{(i)}}(z_{i,1}^{-1},\,\dots,\,z_{i,m_i}^{-1}). 
\end{equation}
Here, for a partition 
$\chi=(\chi_{1} \ge \cdots \ge \chi_{m})$ 
of length less than or equal to $m \in \BZ_{\ge 0}$, 
$s_{\chi}(x)=s_{\chi}(x_{1},\,\dots,\,x_{m})$ denotes 
the Schur polynomial in the variables $x_{1},\,\dots,\,x_{m}$ 
corresponding to the partition $\chi$. 
We can easily show (see \cite[Sect.~7.3]{NS16}) that 
$s_{\brho}(z^{-1})(\Img \Phi_{\lambda}) \subset \Img \Phi_{\lambda}$ 
for each $\brho = (\rho^{(i)})_{i \in I} \in \ol{\Par(\lambda)}$. 
Hence we can define a $U_{\q}'(\Fg_{\af})$-module homomorphism
$z_{\brho}:V(\lambda) \rightarrow V(\lambda)$ in such a way that 
the following diagram commutes:
%
%
\begin{equation} \label{eq:zrho}
\begin{CD}
V(\lambda) @>{\Phi_{\lambda}}>> \ti{V}(\lambda) \\
@V{z_{\brho}}VV @VV{s_{\brho}(z^{-1})}V \\
V(\lambda) @>{\Phi_{\lambda}}>> \ti{V}(\lambda);
\end{CD}
\end{equation}
note that $z_{\brho}v_{\lambda} = S_{\brho}^{-}v_{\lambda}$ 
in the notation of \cite{BN} (and \cite{NS16}). 
Here, recall that $\ti{V}(\lambda)$ has the crystal basis 
$\bigl(\ti{\CL}(\lambda):=
\bigotimes_{i \in I} \CL(\vpi_{i})^{\otimes m_{i}}, 
\ti{\CB}(\lambda):=
\bigotimes_{i \in I} \CB(\vpi_{i})^{\otimes m_{i}}\bigr)$. 
We see from part (iii) above that 
the map $s_{\brho}(z^{-1})$ preserves $\ti{\CL}(\lambda)$. 
Also, we know from \cite[page 369, the 2nd line from below]{BN} that 
$\Phi_{\lambda}(\CL(\lambda)) \subset \ti{\CL}(\lambda)$. 
Therefore, we deduce that the map 
$z_{\brho}:V(\lambda) \rightarrow V(\lambda)$ preserves $\CL(\lambda)$, 
and hence induces a $\BC$-linear map 
$z_{\brho}:\CL(\lambda)/\q\CL(\lambda) \rightarrow \CL(\lambda)/\q\CL(\lambda)$; 
this map commutes with the Kashiwara operators. 
It follows from \cite[p.\,371]{BN} that
%
%
\begin{equation} \label{eq:CBlam}
\CB(\lambda) = \bigl\{
 z_{\brho}b \mid 
 \brho \in \Par(\lambda),\,b \in \CB_{0}(\lambda) \bigr\}, 
\end{equation}
where $\CB_{0}(\lambda)$ denotes the connected component 
of $\CB(\lambda)$ containing $u_{\lambda}$, and
%
%
\begin{equation} \label{eq:par}
\Par(\lambda) := 
 \bigl\{ \brho = (\rho^{(i)})_{i \in I} \mid 
 \text{$\rho^{(i)}$ is a partition of length $< m_{i}$ 
 for each $i \in I$} \bigr\};
\end{equation}
we understand that a partition of length less than $0$ is 
the empty partition $\emptyset$. For $\brho \in \Par(\lambda)$, we set 
%
%
\begin{equation} \label{eq:u-rho}
u^{\brho} := z_{\brho}u_{\lambda} \in \CB(\lambda). 
\end{equation}
%
%
\begin{rem} \label{rem:zGb}
We see from \cite[Theorem~4.16\,(ii)]{BN} (see also the argument after 
\cite[(7.3.8)]{NS16}) that $z_{\brho}G(b) = G(z_{\brho}b)$ for 
$b \in \CB_{0}(\lambda)$ and $\brho \in \ol{\Par(\lambda)}$. 
\end{rem}
%
%
\subsection{Level-zero van der Kallen modules and their graded characters.}
\label{subsec:main}

Let $\lambda \in P^{+} \subset P_{\af}^{0}$, and set 
%
%
\begin{equation} \label{eq:J}
\J=\J_{\lambda}:= 
\bigl\{ i \in I \mid \pair{\lambda}{\alpha_i^{\vee}}=0 \bigr\} \subset I.
\end{equation}
Let $\{ v_{x} \}_{x \in W_{\af}}$ be the family of 
extremal weight vectors in $V(\lambda)$ corresponding to $v_{\lambda}$ 
(see Section~\ref{subsec:extremal}). 
For each $x \in W_{\af}$, we define 
the Demazure submodule $V_{x}^{-}(\lambda)$ of $V(\lambda)$ by
%
%
\begin{equation} \label{eq:dem}
V_{x}^{-}(\lambda):=U_{\q}^{-}(\Fg_{\af})v_{x} \subset V(\lambda). 
\end{equation}
We see that the Demazure submodule $V_{x}^{-}(\lambda)$ has 
the ($\Fh_{\af}$-)weight space decomposition of the form: 
\begin{equation} \label{eq:wsd1}
V_{x}^{-}(\lambda) = 
 \bigoplus_{k \in \BZ}
\Biggl(
   \bigoplus_{\gamma \in Q} 
   V_{x}^{-}(\lambda)_{\lambda+\gamma+k\delta}
\Biggr),
\end{equation}
where each weight space $V_{x}^{-}(\lambda)_{\lambda+\gamma+k\delta}$ 
is finite-dimensional. 
Also, we know from \cite[Sect.~2.8]{Kas05} (see also \cite[Sect.~4.1]{NS16}) 
that $V_{x}^{-}(\lambda)$ is compatible 
with the global basis of $V(\lambda)$, that is, 
there exists a subset $\CB_{x}^{-}(\lambda)$ of 
the crystal basis $\CB(\lambda)$ such that 
%
%
\begin{equation} \label{eq:gb}
V_{x}^{-}(\lambda) = 
\bigoplus_{b \in \CB_{x}^{-}(\lambda)} \BC(\q) G(b) 
\subset
\bigoplus_{b \in \CB(\lambda)} \BC(\q) G(b) = V(\lambda);  
\end{equation}
by \cite[Lemma~5.4.1]{NNS1}, we have
%
%
\begin{equation} \label{eq:cb1}
\CB_{x}^{-}(\lambda) = 
 \bigl\{ z_{\brho}b \mid \brho \in \Par(\lambda),\,
 b \in \CB_{x}^{-}(\lambda) \cap \CB_{0}(\lambda)
\bigr\}.
\end{equation}
%
%
\begin{rem} \label{rem:dem}
By \cite[Lemma~4.1.2]{NS16}, we have
$V_{x}^{-}(\lambda)=V_{\PJ(x)}^{-}(\lambda)$ for $x \in W_{\af}$. 
Also, it follows from \cite[Lemma~5.2.3]{NS16} that 
for $x,y \in (\WJu)_{\af}$, 
%
%
\begin{equation} \label{eq:subset}
V_{y}^{-}(\lambda) \subset V_{x}^{-}(\lambda) \iff 
\CB_{y}^{-}(\lambda) \subset \CB_{x}^{-}(\lambda) \iff 
y \sige x.
\end{equation}
\end{rem}

Now, for $w \in \WJu$, we define a quotient 
$U_{\q}^{-}(\Fg_{\af})$-module (level-zero van der Kallen module) 
of $V_{w}^{-}(\lambda)$ by
%
%
\begin{equation} \label{eq:Kg}
\Kg_{w}^{-}(\lambda):=
 V_{w}^{-}(\lambda)\Biggm/
 \sum_{ z \in \WJu, \, z > w } 
 V_{z}^{-}(\lambda); 
\end{equation}
note that $V_{z}^{-}(\lambda) \subset V_{w}^{-}(\lambda)$ for all 
$z \in \WJu$ such that $z > w$ (see \eqref{eq:subset} and Remark~\ref{rem:SB}).
Also, for each $w \in \WJu$, we define another quotient $U_{\q}^{-}(\Fg)$-module 
of $V_{w}^{-}(\lambda)$ by
%
%
\begin{equation} \label{eq:Kl}
\Kl_{w}^{-}(\lambda):=
 V_{w}^{-}(\lambda)\Biggm/
 \Biggl( 
 \sum_{ z \in \WJu, \, z > w }
 V_{z}^{-}(\lambda) + 
 \sum_{\brho \in \ol{\Par(\lambda)},\,\brho \ne (\emptyset)_{i \in I}}
 z_{\brho} V_{w}^{-}(\lambda) \Biggr),
\end{equation}
where $z_{\brho}:V(\lambda) \rightarrow V(\lambda)$ is as in \eqref{eq:zrho}. 
Here we know from \cite[(5.15) and (5.16)]{NNS1} that if we set 
\begin{equation}
X_{w}^{-}(\lambda):=
 \sum_{\brho \in \ol{\Par(\lambda)},\,\brho \ne (\emptyset)_{i \in I}}
 z_{\brho} V_{w}^{-}(\lambda),
\end{equation}
then $X_{w}^{-}(\lambda) = \bigoplus_{b \in \CB(X_{w}^{-}(\lambda))} \BC(\q)G(b)$, 
where 
%
%
\begin{equation} \label{eq:cb2}
\CB(X_{w}^{-}(\lambda)):=\bigl\{ z_{\brho}b \mid 
 \brho \in \Par(\lambda),\,\brho \ne (\emptyset)_{i \in I},\,
 b \in \CB_{w}^{-}(\lambda) \cap \CB_{0}(\lambda)
\bigr\}.
\end{equation}
By \eqref{eq:cb1}, we see that 
$\CB(X_{w}^{-}(\lambda)) \subset \CB_{w}^{-}(\lambda)$, and hence
$X_{w}^{-}(\lambda) \subset V_{w}^{-}(\lambda)$.
Observe that $\Kg_{w}^{-}(\lambda)$ and $\Kl_{w}^{-}(\lambda)$ have
the ($\Fh_{\af}$-)weight space decompositions induced by that of 
$V_{w}^{-}(\lambda)$ (see \eqref{eq:wsd1}): 
\begin{equation*}
\Kg_{w}^{-}(\lambda) = 
 \bigoplus_{k \in \BZ}
\Biggl(
   \bigoplus_{\gamma \in Q} 
   \Kg_{w}^{-}(\lambda)_{\lambda+\gamma+k\delta}
\Biggr), \qquad
\Kl_{w}^{-}(\lambda) = 
 \bigoplus_{k \in \BZ}
\Biggl(
   \bigoplus_{\gamma \in Q} 
   \Kl_{w}^{-}(\lambda)_{\lambda+\gamma+k\delta}
\Biggr).
\end{equation*}
By putting $q:=\be^{\delta}$, we define
\begin{equation*}
\gch \Kg_{w}^{-}(\lambda):=
\sum_{\gamma \in Q,\,k \in \BZ} 
\bigl( \dim \Kg_{w}^{-}(\lambda)_{\lambda+\gamma+k\delta} \bigr) 
\be^{\lambda+\gamma} q^{k}, 
\end{equation*}
\begin{equation*}
\gch \Kl_{w}^{-}(\lambda):=
\sum_{\gamma \in Q,\,k \in \BZ} 
\bigl( \dim \Kl_{w}^{-}(\lambda)_{\lambda+\gamma+k\delta} \bigr) 
\be^{\lambda+\gamma} q^{k}.
\end{equation*}
The following is the main result of this paper. 
%
%
\begin{thm}[cf. {\cite[Corollaries~3.19 and 3.20]{FKM}}] \label{thm:main}
Let $\lambda \in P^{+}$, and set 
$S=S_{\lambda}:=\bigl\{ i \in I \mid \pair{\lambda}{\alpha_{i}^{\vee}}=0 \bigr\}$. 
For $w \in \WJu$, the graded character $\gch \Kg_{w}^{-}(\lambda)$ can be 
expressed as{\rm:}
%
%
\begin{equation} \label{eq:main1}
\gch \Kg_{w}^{-}(\lambda) = 
\left(
 \prod_{i \in I} 
 \prod_{r=1}^{\pair{\lambda}{\alpha_{i}^{\vee}}-\eps_{i}} (1 - q^{-r}) \right)^{-1}
 E_{w\lambda}(q,\,\infty), 
\end{equation}
where $E_{w\lambda}(q,\infty)$ is the specialization of 
the nonsymmetric Macdonald polynomial $E_{w\lambda}(q,t)$ at $t=\infty$, 
and for $i \in I$, 
%
%
\begin{equation} \label{eq:eps}
\eps_{i}=\eps_{i}(\xcr{w}):=
\begin{cases}
1 & \text{\rm if $\xw s_{i} > \xw$}, \\
0 & \text{\rm if $\xw s_{i} < \xw$}. 
\end{cases}
\end{equation}
Moreover, it holds that
%
%
\begin{equation} \label{eq:main2}
\gch \Kl_{w}^{-}(\lambda) = E_{w\lambda}(q,\,\infty). 
\end{equation}
\end{thm}
%
%
\begin{rem} \label{rem:Iw}
Keep the notation and setting of the theorem above. 
We see by \eqref{eq:mcr} that 
$\xw \alpha_{i} \in \Delta^{-}$ for all $i \in \J$ 
since $\xw = \mcr{w} \lng(\J)$. 
Hence it follows that $\xw s_{i} < \xw$ for all $i \in \J$. 
\end{rem}

We will give a proof of Theorem~\ref{thm:main} in Section~\ref{sec:prf-main}. 
%
%
\section{Semi-infinite Lakshmibai-Seshadri paths}
\label{sec:SLS}
%
%
\subsection{Crystal structure on semi-infinite LS paths.}
\label{subsec:SLS}

In this subsection, we fix $\lambda \in P^{+} \subset P_{\af}^{0}$ 
(see \eqref{eq:P-fin} and \eqref{eq:P}), 
and take $\J=\J_{\lambda}$ as in \eqref{eq:J}. 
%
%
\begin{dfn} \label{dfn:SBa}
For a rational number $0 < a < 1$, 
we define $\SBa$ to be the subgraph of $\SBJ$ 
with the same vertex set but having only 
those directed edges of the form
$x \edge{\beta} y$ for which 
$a\pair{x\lambda}{\beta^{\vee}} \in \BZ$ holds.
\end{dfn}
%
%
\begin{dfn}\label{dfn:SLS}
A semi-infinite Lakshmibai-Seshadri (LS for short) path of 
shape $\lambda $ is a pair 
%
%
\begin{equation} \label{eq:SLS}
\pi = (\bx \,;\, \ba) 
     = (x_{1},\,\dots,\,x_{s} \,;\, a_{0},\,a_{1},\,\dots,\,a_{s}), \quad s \ge 1, 
\end{equation}
of a strictly decreasing sequence $\bx : x_1 \sig \cdots \sig x_s$ 
of elements in $(\WJu)_{\af}$ and an increasing sequence 
$\ba : 0 = a_0 < a_1 < \cdots  < a_s =1$ of rational numbers 
satisfying the condition that there exists a directed path 
from $x_{u+1}$ to  $x_{u}$ in $\SBb{a_{u}}$ 
for each $u = 1,\,2,\,\dots,\,s-1$. 
\end{dfn}

We denote by $\SLS(\lambda)$ 
the set of all semi-infinite LS paths of shape $\lambda$.
Following \cite[Sect.~3.1]{INS} (see also \cite[Sect.~2.4]{NS16}), 
we endow the set $\SLS(\lambda)$ 
with a crystal structure with weights in $P_{\af}$ as follows. 
Let $\pi \in \SLS(\lambda)$ be of the form \eqref{eq:SLS}. 
We define $\ol{\pi}:[0,1] \rightarrow \BR \otimes_{\BZ} P_{\af}$ 
to be the piecewise-linear, continuous map 
whose ``direction vector'' for the interval 
$[a_{u-1},\,a_{u}]$ is $x_{u}\lambda \in P_{\af}$ 
for each $1 \le u \le s$, that is, 
%
%
\begin{equation} \label{eq:olpi}
\ol{\pi} (t) := 
\sum_{k = 1}^{u-1}(a_{k} - a_{k-1}) x_{k}\lambda + (t - a_{u-1}) x_{u}\lambda
\quad
\text{for $t \in [a_{u-1},\,a_u]$, $1 \le u \le s$}. 
\end{equation}
We know from \cite[Proposition~3.1.3]{INS} that $\ol{\pi}$ 
is an (ordinary) LS path of shape $\lambda$, introduced in \cite[Sect.~4]{Lit95}. 
We set
%
%
\begin{equation} \label{eq:wt}
\wt (\pi):= \ol{\pi}(1) = \sum_{u = 1}^{s} (a_{u}-a_{u-1})x_{u}\lambda \in P_{\af}.
\end{equation}
We define root operators $e_{i}$, $f_{i}$, $i \in I_{\af}$, 
in the same manner as in \cite[Sect.~2]{Lit95}. Set 
%
%
\begin{equation} \label{eq:H}
\begin{cases}
H^{\pi}_{i}(t) := \pair{\ol{\pi}(t)}{\alpha_{i}^{\vee}} \quad 
\text{for $t \in [0,1]$}, \\[1.5mm]
m^{\pi}_{i} := 
 \min \bigl\{ H^{\pi}_{i} (t) \mid t \in [0,1] \bigr\}. 
\end{cases}
\end{equation}
As explained in \cite[Remark~2.4.3]{NS16}, 
all local minima of the function $H^{\pi}_{i}(t)$, $t \in [0,1]$, 
are integers; in particular, 
the minimum value $m^{\pi}_{i}$ is a nonpositive integer 
(recall that $\ol{\pi}(0)=0$, and hence $H^{\pi}_{i}(0)=0$).
We define $e_{i}\pi$ as follows. 
If $m^{\pi}_{i}=0$, then we set $e_{i} \pi := \bzero$, 
where $\bzero$ is an additional element not 
contained in any crystal. 
If $m^{\pi}_{i} \le -1$, then we set
%
%
\begin{equation} \label{eq:t-e}
\begin{cases}
t_{1} := 
  \min \bigl\{ t \in [0,\,1] \mid 
    H^{\pi}_{i}(t) = m^{\pi}_{i} \bigr\}, \\[1.5mm]
t_{0} := 
  \max \bigl\{ t \in [0,\,t_{1}] \mid 
    H^{\pi}_{i}(t) = m^{\pi}_{i} + 1 \bigr\}; 
\end{cases}
\end{equation}
notice that $H^{\pi}_{i}(t)$ is 
strictly decreasing on the interval $[t_{0},\,t_{1}]$. 
Let $1 \le p \le q \le s$ be such that 
$a_{p-1} \le t_{0} < a_p$ and $t_{1} = a_{q}$. 
Then we define $e_{i}\pi$ to be
%
%
\begin{equation} \label{eq:epi}
\begin{split}
& e_{i} \pi := ( 
  x_{1},\,\ldots,\,x_{p},\,s_{i}x_{p},\,s_{i}x_{p+1},\,\ldots,\,
  s_{i}x_{q},\,x_{q+1},\,\ldots,\,x_{s} ; \\
& \hspace*{40mm}
  a_{0},\,\ldots,\,a_{p-1},\,t_{0},\,a_{p},\,\ldots,\,a_{q}=t_{1},\,
\ldots,\,a_{s});
\end{split}
\end{equation}
if $t_{0} = a_{p-1}$, then we drop $x_{p}$ and $a_{p-1}$, and 
if $s_{i} x_{q} = x_{q+1}$, then we drop $x_{q+1}$ and $a_{q}=t_{1}$.
Similarly, we define $f_{i}\pi$ as follows. 
Note that $H^{\pi}_{i}(1) - m^{\pi}_{i}$ is a nonnegative integer. 
If $H^{\pi}_{i}(1) - m^{\pi}_{i} = 0$, then we set $f_{i} \pi := \bzero$. 
If $H^{\pi}_{i}(1) - m^{\pi}_{i}  \ge 1$, then we set
%
%
\begin{equation} \label{eq:t-f}
\begin{cases}
t_{0} := 
 \max \bigl\{ t \in [0,1] \mid H^{\pi}_{i}(t) = m^{\pi}_{i} \bigr\}, \\[1.5mm]
t_{1} := 
 \min \bigl\{ t \in [t_{0},\,1] \mid H^{\pi}_{i}(t) = m^{\pi}_{i} + 1 \bigr\};
\end{cases}
\end{equation}
notice that $H^{\pi}_{i}(t)$ is 
strictly increasing on the interval $[t_{0},\,t_{1}]$. 
Let $0 \le p \le q \le s-1$ be such that $t_{0} = a_{p}$ and 
$a_{q} < t_{1} \le a_{q+1}$. Then we define $f_{i}\pi$ to be
%
%
\begin{equation} \label{eq:fpi}
\begin{split}
& f_{i} \pi := ( x_{1},\,\ldots,\,x_{p},\,s_{i}x_{p+1},\,\dots,\,
  s_{i} x_{q},\,s_{i} x_{q+1},\,x_{q+1},\,\ldots,\,x_{s} ; \\
& \hspace{40mm} 
  a_{0},\,\ldots,\,a_{p}=t_{0},\,\ldots,\,a_{q},\,t_{1},\,
  a_{q+1},\,\ldots,\,a_{s});
\end{split}
\end{equation}
if $t_{1} = a_{q+1}$, then we drop $x_{q+1}$ and $a_{q+1}$, and 
if $x_{p} = s_{i} x_{p+1}$, then we drop $x_{p}$ and $a_{p}=t_{0}$.
In addition, we set $e_{i} \bzero = f_{i} \bzero := \bzero$ 
for all $i \in I_{\af}$.
%
%
\begin{thm}[{see \cite[Theorem~3.1.5]{INS}}] \label{thm:SLS}
\mbox{}
\begin{enu}
\item The set $\SLS(\lambda) \sqcup \{ \bzero \}$ is 
stable under the action of the root operators 
$e_{i}$ and $f_{i}$, $i \in I_{\af}$.

\item For each $\pi \in \SLS(\lambda)$ 
and $i \in I_{\af}$, we set 
\begin{equation*}
\begin{cases}
\ve_{i} (\pi) := 
 \max \bigl\{ n \ge 0 \mid e_{i}^{n} \pi \neq \bzero \bigr\}, \\[1.5mm]
\vp_{i} (\pi) := 
 \max \bigl\{ n \ge 0 \mid f_{i}^{n} \pi \neq \bzero \bigr\}.
\end{cases}
\end{equation*}
Then, the set $\SLS(\lambda)$, 
equipped with the maps $\wt$, $e_{i}$, $f_{i}$, $i \in I_{\af}$, 
and $\ve_{i}$, $\vp_{i}$, $i \in I_{\af}$, 
defined above, is a crystal with weights in $P_{\af}$.
\end{enu}
\end{thm}
%
%

Finally, if $\pi \in \SLS(\lambda)$ is of the form \eqref{eq:SLS}, 
then we set 
$\kappa(\pi):=x_{s} \in (\WJu)_{\af}$, 
and call it the final direction of $\pi$. 
For $x \in W_{\af}$, we set
%
%
\begin{equation} \label{eq:SLS-dem}
\SLS_{\sige x}(\lambda) := 
 \bigl\{
   \pi \in \SLS(\lambda) \mid \kappa(\pi) \sige \PJ(x)
 \bigr\}. 
\end{equation}
%
%
\subsection{Realization of the crystal bases of Demazure submodules by semi-infinite LS paths.}
\label{subsec:Parp}

As in the previous subsection, 
we fix $\lambda \in P^{+}$, and take $\J=\J_{\lambda}$ as in \eqref{eq:J}. 
We write $\lambda \in P^{+}$ as 
$\lambda = \sum_{i \in I} m_{i} \varpi_{i}$ 
with $m_{i} \in \BZ_{\ge 0}$ for $i \in I$; 
recall the definitions of 
$\ol{\Par(\lambda)}$ and $\Par(\lambda)$ 
from \eqref{eq:olpar} and \eqref{eq:par}, respectively.
For $\brho = (\rho^{(i)})_{i \in I} \in \ol{\Par(\lambda)}$, we set 
$|\brho|:=\sum_{i \in I} |\rho^{(i)}|$, where for a partition 
$\chi = (\chi_1 \ge \chi_2 \ge \cdots \ge \chi_{m})$, 
we set $|\chi| := \chi_{1}+\cdots+\chi_{m}$. 
We equip the set $\Par(\lambda)$ with a crystal structure as follows: 
for each $\brho = (\rho^{(i)})_{i \in I} \in \Par(\lambda)$, we set
\begin{equation}
\begin{cases}
e_{j} \brho = f_{j} \brho := \bzero, \quad 
\ve_{j} (\brho) = \vp_{j} (\brho) := -\infty 
  & \text{for $j \in I_{\af}$}, \\[1.5mm]
\wt(\brho) := - |\brho| \delta. &
\end{cases}
\end{equation}

Let $\Conn(\SLS(\lambda))$ denote the set of 
all connected components of $\SLS(\lambda)$, 
and let $\SLS_{0}(\lambda) \in \Conn(\SLS(\lambda))$ denote 
the connected component of $\SLS(\lambda)$
containing $\pi_{\lambda}:=(e\,;\,0,\,1) \in \SLS(\lambda)$, 
where $e$ is the identity element of $W_{\af}$.
%
%
\begin{prop} \label{prop:SLS}
Keep the notation and setting above. 
\begin{enu}

\item Each connected component $C \in \Conn(\SLS(\lambda))$ 
of $\SLS(\lambda)$ contains a unique element of the form{\rm:}
%
%
\begin{equation} \label{eq:etaC}
\pi^{C} = 
 (\PJ(t_{\xi_1}),\, 
  \PJ(t_{\xi_2}),\,\dots,\,\PJ(t_{\xi_{s-1}}),\,e \,;\, 
  a_{0},\,a_{1},\,\dots,\,a_{s-1},\,a_{s})
\end{equation}
for some $s \ge 1$ and $\xi_{1},\,\xi_{2},\,\ldots,\,\xi_{s-1} \in Q^{\vee,+}$ 
{\rm (see \cite[Proposition~7.1.2]{INS})}. 

\item There exists a bijection 
$\Theta:\Conn(\SLS(\lambda)) \rightarrow \Par(\lambda)$ such that 
$\wt(\pi^{C})=\lambda-|\Theta(C)|\delta = \lambda+\wt(\Theta(C))$ 
{\rm (see \cite[Proposition~7.2.1 and its proof]{INS})}. 

\item Let $C \in \Conn(\SLS(\lambda))$. 
Then, there exists an isomorphism $C \stackrel{\sim}{\rightarrow}
\bigl\{\Theta(C)\bigr\} \otimes \SLS_{0}(\lambda)$ of crystals that maps 
$\pi^{C}$ to $\Theta(C) \otimes \pi_{\lambda}$. 
Consequently, $\SLS(\lambda)$ is isomorphic as a crystal  
to $\Par(\lambda) \otimes \SLS_{0}(\lambda)$ 
{\rm (see \cite[Proposition~3.2.4 and its proof]{INS})}. 
\end{enu}
\end{prop}

We know the following from 
\cite[Theorem~3.2.1]{INS} and \cite[Theorem~4.2.1]{NS16}. 
%
%
\begin{thm} \label{thm:isom}
There exists an isomorphism 
$\Psi_{\lambda}:\CB(\lambda) \stackrel{\sim}{\rightarrow} \SLS(\lambda)$ 
of crystals satisfying the following conditions {\rm (1)} and {\rm (2)}. 
\begin{enu}
\item $\Psi_{\lambda}(u^{\brho}) = \pi^{\Theta^{-1}(\brho)}$ 
for every $\brho \in \Par(\lambda)$, where $u^{\brho}=z_{\brho}u_{\lambda} \in \CB(\lambda)$ 
is as defined in \eqref{eq:u-rho}. In particular, 
$\Psi_{\lambda}(u_{\lambda}) = \pi_{\lambda}$. 

\item $\Psi_{\lambda}(\CB_{x}^{-}(\lambda)) = \SLS_{\sige x}(\lambda)$ 
for every $x \in (\WJu)_{\af}$. 
\end{enu}
\end{thm}
%
%
\begin{rem} \label{rem:weyl}
Recall from \eqref{eq:W1} that the crystal basis $\CB(\lambda)$ has 
an action of the affine Weyl group $W_{\af}$. 
Hence, by Theorem~\ref{thm:isom}, 
the crystal $\SLS(\lambda)$ also has the induced action of 
the affine Weyl group $W_{\af}$, which we denote by 
$x \cdot \pi$ for $x \in W_{\af}$ and $\pi \in \SLS(\lambda)$.
\end{rem}

For $\gamma \in Q$ and $k \in \BZ$, we set 
$\fwt(\lambda+\gamma+k\delta):=\lambda+\gamma \in P$ and 
$\qwt(\lambda+\gamma+k\delta):=k \in \BZ$. Let $w \in \WJu$. 
Then, from Theorem~\ref{thm:isom} and \eqref{eq:gb}, 
we deduce that
%
%
\begin{equation} \label{eq:gch2}
\gch \Kg_{w}^{-}(\lambda) = 
\sum_{\pi \in \Kc{w}(\lambda)} \be^{\fwt(\wt(\pi))} q^{\qwt(\wt(\pi))}, 
\end{equation}
where 
%
%
\begin{equation} \label{eq:BK}
\Kc{w}(\lambda):=
\SLS_{\sige w}(\lambda) \setminus 
  \bigcup_{z \in \WJu,\,z > w} \SLS_{\sige z}(\lambda). 
\end{equation}
If we set 
$( (\WJu)_{\af} )_{\sige x}:=\bigl\{y \in (\WJu)_{\af} \mid y \sige x\bigr\}$ 
for $x \in (\WJu)_{\af}$, and 
%
%
\begin{equation} \label{eq:kw}
\kq{w} := ( (\WJu)_{\af} )_{\sige w} \setminus 
\bigcup_{z \in \WJu,\,z > w} ( (\WJu)_{\af} )_{\sige z}, 
\end{equation}
then it is easily verified that 
%
%
\begin{equation} \label{eq:BK2}
\Kc{w}(\lambda)=
\bigl\{ \pi \in \SLS(\lambda) \mid \kappa(\pi) \in \kq{w} \bigr\}.
\end{equation}
%
%
\section{Description of the crystal bases of level-zero van der Kallen modules.}
\label{sec:kappa}
%
%
\subsection{Quantum Bruhat graph and the tilted Bruhat order.}
\label{subsec:QBG}

In this subsection, we take and fix a subset $\J$ of $I$. 

\begin{dfn}
The (parabolic) quantum Bruhat graph $\QBJ$ is 
the ($\Delta^{+} \setminus \DeJ^{+})$-labeled
directed graph whose vertices are the elements of $\WJu$, and 
whose directed edges are of the form: $u \edge{\beta} v$ 
for $u,v \in \WJu$ and $\beta \in \Delta^{+} \setminus \DeJ^{+}$ 
such that $v= \mcr{us_{\beta}}$, and such that either of 
the following (i) or (ii) holds:
\begin{enu}
\item[(i)] $\ell(v) = \ell (u) + 1$; 
\item[(ii)] $\ell(v) = \ell (u) + 1 - 2 \pair{\rho-\rho_{\J}}{\beta^{\vee}}$.
\end{enu}
An edge satisfying (i) (resp., (ii)) is called a Bruhat (resp., quantum) edge. 
When $\J=\emptyset$, we write $\QB$ for $\mathrm{QBG}(W^{\emptyset})$. 
\end{dfn}
%
%
\begin{rem} \label{rem:PQBG}
We know from \cite[Remark~6.13]{LNSSS} that for each $u,\,v \in \WJu$, 
there exists a directed path in $\QBJ$ from $u$ to $v$.
\end{rem}

Let $u,\,v \in \WJu$, and let 
$\bp:u=
 u_{0} \edge{\beta_{1}}
 u_{1} \edge{\beta_{2}} \cdots 
       \edge{\beta_{s}}
 u_{s}=v$
be a directed path in $\QBJ$ from $u$ to $v$. 
Then we define the weight of $\bp$ by
%
%
\begin{equation} \label{eq:wtdp}
\wt^{\J}(\bp) := \sum_{
 \begin{subarray}{c}
 1 \le r \le s\,; \\[1mm]
 \text{$u_{r-1} \edge{\beta_{r}} u_{r}$ is} \\[1mm]
 \text{a quantum edge}
 \end{subarray}}
\beta_{r}^{\vee} \in Q^{\vee,+}; 
\end{equation}
when $\J=\emptyset$, we write $\wt(\bp)$ for $\wt^{\emptyset}(\bp)$. 
We know the following proposition from 
\cite[Proposition~8.1 and its proof]{LNSSS}.
%
%
\begin{prop} \label{prop:81}
Let $u,\,v \in \WJu$. 
Let $\bp$ be a shortest directed path in $\QBJ$ from $u$ to $v$, 
and $\bq$ an arbitrary directed path in $\QBJ$ from $u$ to $v$. 
Then, $[\wt^{\J}(\bq)-\wt^{\J}(\bp)]^{\J} \in Q_{\Jc}^{\vee,+}$, 
where $[\,\cdot\,]^{\J} : Q^{\vee} \twoheadrightarrow Q_{\Jc}^{\vee}$ 
is as defined in Section~\ref{subsec:SiBG}. Moreover, $\bq$ is also shortest 
if and only if $[\wt^{\J}(\bq)]^{\J}=[\wt^{\J}(\bp)]^{\J}$.
\end{prop}

For $u,\,v \in \WJu$, we take a shortest directed path $\bp$ in 
$\QBJ$ from $u$ to $v$, and set 
$\wt^{\J}(u \Rightarrow v):=[\wt^{\J}(\bp)]^{\J} \in Q_{\Jc}^{\vee,+}$. 
When $\J=\emptyset$, we write $\wt(u \Rightarrow v)$ 
for $\wt^{\emptyset}(u \Rightarrow v)$. 
%
%
\begin{lem}[{\cite[Lemma~7.2]{LNSSS2}}] \label{lem:wtS} \mbox{}
Let $u,\,v \in \WJu$, and let $u_{1} \in u\WJ$, $v_{1} \in v\WJ$. 
Then we have $\wt^{\J}(u \Rightarrow v) = [\wt(u_{1} \Rightarrow v_{1})]^{S}$. 
\end{lem}
For $u,\,v \in W$, we denote by $\ell(u \Rightarrow v)$ 
the length of a shortest directed path from $u$ to $v$ in $\QB$. 
%
%
\begin{lem}[{\cite[Lemma~7.7]{LNSSS}}] \label{lem:dia-qb}
Let $u,v \in W$, and $i \in I$. 
\begin{enu}
\item If $u^{-1}\alpha_{i} \in \Delta^{+}$ and $v^{-1}\alpha_{i} \in \Delta^{-}$, then 
$\ell(u \Rightarrow v) = \ell(s_{i}u \Rightarrow v) + 1 = \ell(u \Rightarrow s_{i}v) + 1$, 
and $\wt(u \Rightarrow v)=\wt(u \Rightarrow s_{i}v) = \wt(s_{i}u \Rightarrow v)$. 

\item If $u^{-1}\alpha_{i},\,v^{-1}\alpha_{i} \in \Delta^{+}$, 
or if $u^{-1}\alpha_{i},\,v^{-1}\alpha_{i} \in \Delta^{-}$, then 
$\ell(u \Rightarrow v) = \ell(s_{i}u \Rightarrow s_{i}v)$, and 
$\wt(u \Rightarrow v)=\wt(s_{i}u \Rightarrow s_{i}v)$. 
\end{enu}
\end{lem}
%
%
\begin{dfn}[\cite{BFP}] \label{dfn:tilted}
For each $w \in W$, we define the $w$-tilted Bruhat order $\le_{w}$ on $W$ as follows:
for $u,v \in W$, 
%
%
\begin{equation} \label{eq:tilted}
u \le_{w} v \iff \ell(w \Rightarrow v) = \ell(w \Rightarrow u) + \ell(u \Rightarrow v).
\end{equation}
Namely, $u \le_{w} v$ if and only if the concatenation of a shortest directed path 
from $w$ to $u$ and one from $u$ to $v$ is one from $w$ to $v$. 
\end{dfn}
%
%
\subsection{Some lemmas on the semi-infinite Bruhat order.}
\label{subsec:lemma}

In this subsection, we fix a subset $S$ of $I$. 
%
%
\begin{lem}[{\cite[Lemma~6.1.1]{INS}}] \label{lem:INS}
If $x,\,y \in W_{\af}$ satisfy $x \sige y$ in $W_{\af}$, 
then $\PJ(x) \sige \PJ(y)$ in $(\WJu)_{\af}$.
\end{lem}
%
%
\begin{lem}[{\cite[Lemmas 4.3.5, 4.3.6, and 4.3.7]{NNS1}}] \label{lem:NNS}
Let $u,v \in \WJu$, and $\xi,\zeta \in Q^{\vee}$. 
Then, 
\begin{equation}
u\PJ(t_{\xi}) \sige v\PJ(t_{\zeta}) \iff 
[\xi]^{\J} \ge \wt^{\J}(v \Rightarrow u) + [\zeta]^{\J}. 
\end{equation}
\end{lem}
%
%
\begin{lem}[{\cite[Lemma~2.3.6]{INS}}] \label{lem:236}
Let $x \in (\WJu)_{\af}$, and $i \in I_{\af}$. 
Then, $x^{-1}\alpha_{i} \not\in (\DeJ)_{\af}$ 
if and only if $s_{i}x \in (\WJu)_{\af}$. 
In particular, for $u \in \WJu$ and $i \in I$, 
$u^{-1}\alpha_{i} \not\in \DeJ$ 
if and only if $s_{i}u \in \WJu$
{\rm(}see also {\rm\cite[Proposition~5.10]{LNSSS})}. 
\end{lem}
%
%
\begin{rem} \label{rem:236}
Keep the notation in Lemma~\ref{lem:236}. 
Let $\lambda \in P^{+} \subset P_{\af}^{0}$ 
be such that $\bigl\{i \in I \mid 
\pair{\lambda}{\alpha_{i}^{\vee}}=0 \bigr\} = S$. 
Then, $x^{-1}\alpha_{i} \not\in (\DeJ)_{\af}$ 
(resp., $u^{-1}\alpha_{i} \not\in \DeJ$) if and only if 
$\pair{x\lambda}{\alpha_{i}^{\vee}} \ne 0$ 
(resp., $\pair{u\lambda}{\alpha_{i}^{\vee}} \ne 0$). 
\end{rem}
%
%
\begin{lem}[{\cite[Lemma~2.3.6]{NS16}}] \label{lem:dia-si}
Let $x,y \in (\WJu)_{\af}$ be such that $x \sile y$, and let $i \in I_{\af}$. 
\begin{enu}
\item If $x^{-1}\alpha_{i} \in (\Delta^{+} \setminus \DeJ^{+})+\BZ\delta$ and 
$y^{-1}\alpha_{i} \in (\Delta^{-} \cup \DeJ)+\BZ\delta$
{\rm(}or equivalently, $\pair{x\lambda}{\alpha_{i}^{\vee}} > 0$ and 
$\pair{y\lambda}{\alpha_{i}^{\vee}} \le 0$, where 
$\lambda$ is as in Remark~\ref{rem:236}{\rm)}, 
then $s_{i}x \in (\WJu)_{\af}$ and $s_{i}x \sile y$. 

\item If $x^{-1}\alpha_{i} \in (\Delta^{+} \cup \DeJ)+\BZ\delta$ and 
$y^{-1}\alpha_{i} \in (\Delta^{-} \setminus \DeJ^{-})+\BZ\delta$
{\rm(}or equivalently, $\pair{x\lambda}{\alpha_{i}^{\vee}} \ge 0$ and 
$\pair{y\lambda}{\alpha_{i}^{\vee}} < 0${\rm)}, 
then $s_{i}y \in (\WJu)_{\af}$ and $x \sile s_{i}y$. 

\item If $x^{-1}\alpha_{i},\,y^{-1}\alpha_{i} \in (\Delta^{+} \setminus \DeJ^{+})+\BZ\delta$ 
{\rm(}or equivalently, $\pair{x\lambda}{\alpha_{i}^{\vee}} > 0$ and 
$\pair{y\lambda}{\alpha_{i}^{\vee}} > 0${\rm)}, or 
if $x^{-1}\alpha_{i},\,y^{-1}\alpha_{i} \in (\Delta^{-} \setminus \DeJ^{-})+\BZ\delta$
{\rm(}or equivalently, $\pair{x\lambda}{\alpha_{i}^{\vee}} < 0$ and 
$\pair{y\lambda}{\alpha_{i}^{\vee}} < 0${\rm)}, then $s_{i}x \sile s_{i}y$. 
\end{enu}
\end{lem}

By Remark~\ref{rem:SB}, the next well-known lemma 
(see, e.g., \cite[Lemma on page~151]{Hum90} for the case of $S=\emptyset$) 
also follows from Lemma~\ref{lem:dia-si} as a special case. 
%
%
\begin{lem} \label{lem:dia-b}
Let $u,v \in \WJu$ be such that $u \sile v$, and let $i \in I$. 
\begin{enu}
\item If $u^{-1}\alpha_{i} \in \Delta^{+} \setminus \DeJ^{+}$ and 
$v^{-1}\alpha_{i} \in \Delta^{-} \cup \DeJ$ 
{\rm(}or equivalently, $\pair{u\lambda}{\alpha_{i}^{\vee}} > 0$ and 
$\pair{v\lambda}{\alpha_{i}^{\vee}} \le 0$, where 
$\lambda$ is as in Remark~\ref{rem:236}{\rm)}, 
then $s_{i}u \in \WJu$ and $s_{i}u \le v$. 

\item If $u^{-1}\alpha_{i} \in \Delta^{+} \cup \DeJ$ and 
$v^{-1}\alpha_{i} \in \Delta^{-} \setminus \DeJ^{-}$
{\rm(}or equivalently, $\pair{u\lambda}{\alpha_{i}^{\vee}} \ge 0$ and 
$\pair{v\lambda}{\alpha_{i}^{\vee}} < 0${\rm)}, 
then $s_{i}y \in \WJu$ and $u \le s_{i}v$. 

\item If $u^{-1}\alpha_{i},\,v^{-1}\alpha_{i} \in \Delta^{+} \setminus \DeJ^{+}$
{\rm(}or equivalently, $\pair{u\lambda}{\alpha_{i}^{\vee}} > 0$ and 
$\pair{v\lambda}{\alpha_{i}^{\vee}} > 0${\rm)}, or 
if $u^{-1}\alpha_{i},\,v^{-1}\alpha_{i} \in \Delta^{-} \setminus \DeJ^{-}$
{\rm(}or equivalently, $\pair{u\lambda}{\alpha_{i}^{\vee}} < 0$ and 
$\pair{v\lambda}{\alpha_{i}^{\vee}} < 0${\rm)}, then $s_{i}u \le s_{i}v$. 
\end{enu}
\end{lem}
%
%
\subsection{Definition and a recursive description of subsets $\EQB(w)$.}
\label{subsec:EQB}

For $w \in W$, we define the right descent set  $I_{w}$ for $w$ by
%
%
\begin{equation} \label{eq:Iw}
I_{w}:=\bigl\{ j \in I \mid ws_{j} < w \bigr\}=
\bigl\{ j \in I \mid w\alpha_{j} \in \Delta^{-} \bigr\}. 
\end{equation}
%
%
\begin{rem}[see Remark~\ref{rem:Iw}] \label{rem:Iw1}
Let $\J$ be a subset of $I$. 
Then we have $\J \subset \Iw$ for $w \in \WJu$. 
\end{rem}

%
\begin{lem}[{\cite[Lemma~3.1.1]{NNS2}}] \label{lem:311}
Let $w \in W$ and $i \in I$ be such that $s_{i}w < w$. 
\begin{enu}
\item[\rm (a)]
$s_{i}w \notin wW_{I_{w}}$ if and only if $-w^{-1}\alpha_{i}$ is not a simple root. 
In this case, $I_{s_{i}w}=I_{w}$. 

\item[\rm (b)]
$s_{i}w \in wW_{I_{w}}$ if and only if $-w^{-1}\alpha_{i}$ is a simple root. 
In this case, $I_{s_{i}w}=I_{w} \setminus \{k\}$, where 
$\alpha_{k}=- w^{-1}\alpha_{i}$. 
\end{enu}
\end{lem}

Now, we recall from \cite[Sect.~3.2]{NNS1} and \cite[Sect.~2.2]{NNS2} 
the definition of the subsets $\EQB(w) \subset W$ for $w \in W$.
For each $w \in W$, take and fix a reduced expression 
\begin{equation} \label{eq:redw}
w=s_{i_{1}}s_{i_{2}} \cdots s_{i_{p}}.
\end{equation}
Then we take $i_{-q},\,i_{-q+1},\,\dots,\,i_{0} \in I$ 
in such a way that $s_{i_{-q}} \cdots s_{i_{0}}s_{i_{1}} \cdots s_{i_{p}}$ 
is a reduced expression for the longest element $\lng \in W$. We set
\begin{equation} \label{eq:beta0}
\beta_{k}:=s_{i_{p}} \cdots s_{i_{k+1}}\alpha_{i_{k}} \in \Delta^{+} \qquad 
\text{for $-q \le k \le p$}.
\end{equation}
Since $\beta_{-q} < \cdots < \beta_{p}$ is a reflection order on $\Delta^{+}$, 
we know (see, e.g., \cite[Theorem~7.3]{LNSSS}) that for each $u \in W$, 
there exists a unique shortest directed path 
\begin{equation} \label{eq:wu0}
w=x_{0} \edge{\beta_{j_1}} x_{1} \edge{\beta_{j_2}} \cdots \edge{\beta_{j_s}} x_{s} = u
\end{equation}
in $\QB$ from $w$ to $u$ such that 
$-q \le j_{1} < j_{2} < \cdots < j_{s} \le p$, 
which we call the ``label-increasing'' directed path. 
The subset $\EQB(w) \subset W$ is defined to be the set of all those elements $u \in W$ 
whose label-increasing directed path \eqref{eq:wu0} from $w$ to $u$ 
satisfies $j_{1} \ge 1$; we know from \cite[Proposition~3.2.5]{NNS1} that 
this definition of $\EQB(w)$ does not depend on the choice of 
a reduced expression \eqref{eq:redw} for $w$. 
The subsets $\EQB(w)$, $w \in W$, above are also determined through 
the following recursive formula by descending induction. 
%
%
\begin{prop}[{\cite[Proposition~3.2.3]{NNS2}}] \label{prop:323}
\mbox{}
\begin{enu}
\item For the longest element $\lng \in W$, 
it holds that $\EQB(\lng)=W$. 

\item Let $w \in W$ and $i \in I$ be such that $s_{i}w < w$. 
If $s_{i}w \notin wW_{I_{w}}$, then 
\begin{equation} \label{eq:2a}
\begin{cases}
\EQB(w) \cap \EQB(s_{i}w) = \emptyset, & \\[1.5mm]
\EQB(w) \cup s_{i}\EQB(w) = \EQB(w) \sqcup \EQB(s_{i}w); &
\end{cases}
\end{equation}
if $s_{i}w \in wW_{I_{w}}$, then 
\begin{equation} \label{eq:2b}
\begin{cases}
\EQB(s_{i}w) = \bigl\{ v \in \EQB(w) \mid s_{i}w \le_{w} v \bigr\}, & \\[1.5mm]
s_{i}\EQB(w) = \EQB(w),
\end{cases}
\end{equation}
where $\le_{w}$ is the $w$-tilted Bruhat order on $W$ 
{\rm(}see Definition~\ref{dfn:tilted}{\rm)}.
\end{enu}
\end{prop}
%
%
\subsection{Description of $\kq{w}$ in terms of the quantum Bruhat graph.}
\label{subsec:kw}
In the case $\J = \emptyset$, for $w \in W=W^{\emptyset}$, we write $\kp{w}$ 
for the subset $K_{w}^{\emptyset}$ defined by \eqref{eq:kw}, that is, 
%
%
\begin{equation} \label{eq:kw2}
\kp{w} = (W_{\af})_{\sige w} \setminus 
\bigcup_{z \in W,\,z > w} (W_{\af})_{\sige z}. 
\end{equation}
%
%
\begin{prop} \label{prop:kappa}
For $w \in W$, the subset $\kp{w}$ is 
identical to the set $\bigl\{ u t_{\xi} \in W_{\af} \mid 
u \in \EQB(w),\,\xi \in \wt(w \Rightarrow u) + Q_{I_{w}}^{\vee,+} \bigr\}$, 
where $\EQB(w)$, $\wt(w \Rightarrow u)$, and $I_{w}$ are as defined in 
Section~{\rm \ref{subsec:EQB}}, \eqref{eq:wtdp}, and \eqref{eq:Iw}, 
respectively. 
\end{prop}

In order to prove Proposition~\ref{prop:kappa}, 
we need some lemmas. 
%
%
\begin{lem} \label{lem:w0}
If $w = \lng$, then $\kp{\lng}= 
\bigl\{u t_{\xi} \in W_{\af} \mid 
u \in W,\,\xi \in Q^{\vee},\,\xi \ge \wt(\lng \Rightarrow u) \bigr\}$. 
\end{lem}
\begin{proof}
The assertion follows immediately 
from Lemma~\ref{lem:NNS} (with $\J=\emptyset$), 
since there is no $z \in W$ such that $z > \lng$. 
\end{proof}
%
%
\begin{lem} \label{lem:tec1}
Let $w \in W$. We have
%
%
\begin{equation} \label{eq:tec1}
\EQB(w) = \bigl\{ u \in W \mid 
\text{\rm there does not exist $z \in W$ such that $z > w$ and $z \le_{w} u$} \bigr\}. 
\end{equation}
\end{lem}

\begin{proof}
Let $u$ be an element of the set on the right-hand side in \eqref{eq:tec1}. 
Define $\beta_{k}$, $-q \le k \le p$, as in \eqref{eq:beta0}, and take 
the label-increasing directed path 
\begin{equation} \label{eq:wu1}
w=x_{0} \edge{\beta_{j_1}} x_{1} \edge{\beta_{j_2}} \cdots \edge{\beta_{j_s}} x_{s} = u
\end{equation}
in $\QB$ from $w$ to $u$ (see \eqref{eq:wu0}). 
Suppose, for a contradiction, that $j_{1} \le 0$. Then we see that 
\begin{equation*}
w\beta_{j_1} = (s_{i_1} \cdots s_{i_p})
(s_{i_p} \cdots s_{i_1}s_{i_0} \cdots s_{i_{j_1+1}}\alpha_{i_{j_1}})
=s_{i_0} \cdots s_{i_{j_1+1}}\alpha_{i_{j_1}} \in \Delta^{+},
\end{equation*}
which implies that $w=x_{0} \edge{\beta_{j_1}} x_{1}$ is a Bruhat edge, and hence $z:=x_{1} > w$. 
Since \eqref{eq:wu1} is a shortest directed path from $w$ to $u$ passing through $z=x_{1}$, 
we see that $z \le_{w} u$, which is a contradiction. 
Thus we obtain $j_{1} \ge 1$, and hence $u \in \EQB(w)$. 

Next, we prove the opposite inclusion $\subset$ by descending induction on $\ell(w)$. 
If $w=\lng$, then the assertion is obvious since 
there does not exist $z \in W$ such that $z > \lng$. 
Let $w \in W$ and $i \in I$ be such that $s_{i}w < w$; 
note that $w^{-1}\alpha_{i} \in \Delta^{-}$. 
Then we can take a reduced expression \eqref{eq:redw} for $w$ 
such that $i_{1}=i$; in this case, $\beta_{1}=-w^{-1}\alpha_{i}$. 
Assume that \eqref{eq:tec1} holds for $w$ (the induction hypothesis), 
and suppose, for a contradiction, that 
for some $u \in \EQB(s_{i}w)$, 
there exists $z \in W$ such that $z > s_{i}w$ and $z \le_{s_{i}w} u$; note that
%
%
\begin{equation} \label{eq:tzu}
\ell(s_{i}w \Rightarrow u) = \ell(s_{i}w \Rightarrow z) + \ell(z \Rightarrow u).
\end{equation}
%
\paragraph{Case 1.}
%
Assume that $s_{i}u > u$, or equivalently, $u^{-1}\alpha_{i} \in \Delta^{+}$. 
We see from Proposition~\ref{prop:323}\,(2) that $s_{i}u \in \EQB(w)$. 
Indeed, if $s_{i}w \notin wW_{\Iw}$, then it follows from 
the second equality of \eqref{eq:2a} that 
$u \in \EQB(w)$ or $s_{i}u \in \EQB(w)$. However, 
since $\EQB(w) \cap \EQB(s_{i}w) = \emptyset$ 
by the first equality of \eqref{eq:2a}, and 
since $u \in \EQB(s_{i}w)$, we obtain $u \notin \EQB(w)$, and hence 
$s_{i}u \in \EQB(w)$. Also, if $s_{i}w \in wW_{\Iw}$, 
then we have $u \in \EQB(s_{i}w) \subset \EQB(w)$ 
by the first equality of \eqref{eq:2b}. Hence we have 
$s_{i}u \in \EQB(w)$ by the second equality of \eqref{eq:2b}.

\paragraph{Subcase 1.1.}
%
Assume that $s_{i}z < z$, or equivalently, $z^{-1}\alpha_{i} \in \Delta^{-}$. 
By Lemmas~\ref{lem:dia-b}\,(1) and \ref{lem:dia-qb}\,(1), we deduce that 
$z \ge w$ and $\ell(w \Rightarrow z) = \ell(s_{i}w \Rightarrow z) - 1$.
Since $u^{-1}\alpha_{i} \in \Delta^{+}$, we have a Bruhat edge
$u \edge{u^{-1}\alpha_{i}} s_{i}u$, which implies that 
$\ell(z \Rightarrow s_{i}u) \le \ell(z \Rightarrow u) +1$. 
Also, since $u^{-1}\alpha_{i} \in \Delta^{+}$ and 
$(s_{i}w)^{-1}\alpha_{i} \in \Delta^{+}$, 
it follows from Lemma~\ref{lem:dia-qb}\,(2) that
$\ell(s_{i}w \Rightarrow u) = \ell(w \Rightarrow s_{i}u)$. 
Combining these, we see that
\begin{align*}
\ell(s_{i}w \Rightarrow u) & = 
\ell(w \Rightarrow s_{i}u) \le 
\ell(w \Rightarrow z) + \ell(z \Rightarrow s_{i}u) \\
& 
\le \ell(s_{i}w \Rightarrow z)-1 + \ell(z \Rightarrow u) +1 
= \ell(s_{i}w \Rightarrow z) + \ell(z \Rightarrow u) \\
& = \ell(s_{i}w \Rightarrow u) \quad \text{by \eqref{eq:tzu}}. 
\end{align*}
In particular, we obtain 
$\ell(w \Rightarrow s_{i}u)  =  
\ell(w \Rightarrow z) + \ell(z \Rightarrow s_{i}u)$, 
which implies that $z \le_{w} s_{i}u$. 
Here we recall that $s_{i}u \in \EQB(w)$ and $z \ge w$, as seen above. 
Therefore, by the induction hypothesis, we must have $z=w$. 
In particular, we obtain $w \le_{s_{i}w} u$. 

By concatenating the label-increasing (shortest) directed path \eqref{eq:wu1} 
in $\QB$ from $w$ to $u$ with the Bruhat edge $s_{i}w \edge{\beta_{1}} w$, 
we obtain
\begin{equation} \label{eq:swu1}
s_{i}w \edge{\beta_{1}} 
w=x_{0} \edge{\beta_{j_1}} x_{1} \edge{\beta_{j_2}} \cdots \edge{\beta_{j_s}} x_{s} = u, 
\end{equation}
which is a shortest directed path from $s_{i}w$ to $u$ 
since $w \le_{s_{i}w} u$. Let 
\begin{equation} \label{eq:swu2}
s_{i}w = y_{0} \edge{\beta_{k_1}} y_{1} 
\edge{\beta_{k_2}} \cdots \edge{\beta_{k_{s+1}}} y_{s+1} = u
\end{equation}
be the shortest directed path from $s_{i}w$ to $u$ 
such that $-q \le k_{1} < \cdots < k_{s+1} \le p$. 
Since $u \in \EQB(s_{i}w)$ by our assumption, 
we see by \cite[Remark~23]{NNS2} that $k_{1} \ge 2$, 
and hence $\beta_{1} < \beta_{k_1}$ in our fixed reflection order. 
However, we know (see, e.g., \cite[Theorem~7.3]{LNSSS}) that 
the shortest directed path \eqref{eq:swu2} is lexicographically minimal, 
that is, for every shortest directed path 
$s_{i}w = y_{0}' \edge{\gamma_{1}} y_{1}'
\edge{\gamma_{2}} \cdots \edge{\gamma_{s+1}} y_{s+1}' = u$
in $\QB$ from $s_{i}w$ to $u$, 
there exists $1 \le a \le s+1$ such that 
$\gamma_{b}=\beta_{k_{b}}$ for $1 \le b \le a-1$, and 
$\gamma_{a} > \beta_{k_{a}}$. Therefore, we obtain 
$\beta_{k_{1}} \le \beta_{1}$, which is a contradiction.

\paragraph{Subcase 1.2.}
%
Assume that $s_{i}z > z$, or equivalently, $z^{-1}\alpha_{i} \in \Delta^{+}$. 
By Lemma~\ref{lem:dia-b}\,(3), we see that $s_{i}z > w$. 
Also, we deduce from Lemma~\ref{lem:dia-qb}\,(2) that 
$\ell(s_{i}w \Rightarrow z) = \ell(w \Rightarrow s_{i}z)$, 
$\ell(s_{i}w \Rightarrow u) = \ell(w \Rightarrow s_{i}u)$, and 
$\ell(z \Rightarrow u) = \ell(s_{i}z \Rightarrow s_{i}u)$. 
Substituting these equalities into \eqref{eq:tzu}, we obtain 
$\ell(w \Rightarrow s_{i}u) = 
\ell(w \Rightarrow s_{i}z) + \ell(s_{i}z \Rightarrow s_{i}u)$, 
which implies that $s_{i}z \le_{w} s_{i}u$. 
Since $s_{i}u \in \EQB(w)$ and $s_{i}z > w$, as seen above, 
the inequality $s_{i}z \le_{w} s_{i}u$ 
contradicts the induction hypothesis. 

\paragraph{Case 2.}
%
Assume that $s_{i}u < u$, or equivalently, $u^{-1}\alpha_{i} \in \Delta^{-}$. 
Then we have $s_{i}w \in wW_{I_{w}}$. Indeed, 
since $s_{i}w < w$, $s_{i}(s_{i}u) = u > s_{i}u$, 
and $s_{i}u \in \EQB(w)$, as seen above, it follows 
from \cite[Lemma~24\,(2)]{NNS2} that $s_{i}(s_{i}u) = u \in \EQB(w)$. 
Thus we have $u \in \EQB(w) \cap \EQB(s_{i}w)$, 
which contradicts Proposition~\ref{prop:323}\,(2). 
Hence we conclude that $s_{i}w \in wW_{I_{w}}$. In particular, 
we obtain $u \in \EQB(w)$ by Proposition~\ref{prop:323}\,(2). 

\paragraph{Subcase 2.1.}
%
Assume that $s_{i}z < z$, or equivalently, $z^{-1}\alpha_{i} \in \Delta^{-}$. 
By Lemmas~\ref{lem:dia-b}\,(1) and \ref{lem:dia-qb}\,(1), 
we see that $z \ge w$ and 
$\ell(w \Rightarrow z) = \ell(s_{i}w \Rightarrow z) - 1$. 
Similarly, we deduce by Lemma~\ref{lem:dia-qb}\,(1) that
$\ell(w \Rightarrow u) = \ell(s_{i}w \Rightarrow u) -1$. 
Substituting these equalities into \eqref{eq:tzu}, 
we obtain $\ell(w \Rightarrow u) = \ell(w \Rightarrow z) + \ell(z \Rightarrow u)$, 
which implies that $z \le_{w} u$. 
Since $u \in \EQB(w)$ and $z \ge w$, as seen above, 
it follows from the induction hypothesis that $z=w$. 
In particular, we have $w \le_{s_{i}w} u$. 
In exactly the same way as in the second paragraph of Subcase 1.1, 
we obtain a contradiction from this inequality. 

\paragraph{Subcase 2.2.}
%
Assume that $s_{i}z > z$, or equivalently, $z^{-1}\alpha_{i} \in \Delta^{+}$. 
By Lemmas~\ref{lem:dia-b}\,(3) and \ref{lem:dia-qb}\,(2), 
we see that $s_{i}z > w$ and 
$\ell(w \Rightarrow s_{i}z) = \ell(s_{i}w \Rightarrow z)$. 
Also, we deduce by Lemma~\ref{lem:dia-qb}\,(1) that
$\ell(w \Rightarrow u) = \ell(s_{i}w \Rightarrow u) -1$, and 
$\ell(s_{i}z \Rightarrow u) = \ell(z \Rightarrow u) -1$. 
Substituting these equalities into \eqref{eq:tzu}, 
we obtain $\ell(w \Rightarrow u) = 
\ell(w \Rightarrow s_{i}z) + \ell(s_{i}z \Rightarrow u)$, 
which implies that $s_{i}z \le_{w} u$. 
Since $u \in \EQB(w)$ and $s_{i}z > w$, as seen above, 
the inequality $s_{i}z \le_{w} u$ contradicts the induction hypothesis. 

This completes the proof of the lemma. 
\end{proof}
%
%
\begin{lem} \label{lem:tec2}
Let $w \in W$ and $i \in I$ be such that $s_{i}w < w$. 
For every $u \in \EQB(s_{i}w)$, we have $s_{i}u > u$, 
or equivalently, $u^{-1}\alpha_{i} \in \Delta^{+}$. 
\end{lem}

\begin{proof}
Suppose, for a contradiction, that 
there exists $u \in \EQB(s_{i}w)$ such that $s_{i}u < u$, or equivalently, 
$u^{-1}\alpha_{i} \in \Delta^{-}$. Since $s_{i}w < w$, 
it follows from Lemma~\ref{lem:dia-qb}\,(1) that 
$\ell(s_{i}w \Rightarrow u) = 
\ell(w \Rightarrow u)+1$. Also, we have a Bruhat edge 
$s_{i}w \edge{-w^{-1}\alpha_{i}} w$, and hence 
$\ell(s_{i}w \Rightarrow w) = 1$. 
Therefore, we obtain $\ell(s_{i}w \Rightarrow w) + \ell(w \Rightarrow u) = 
\ell(s_{i}w \Rightarrow u)$, which implies that $w \le_{s_iw} u$. 
Since $w > s_{i}w$ and $u \in \EQB(s_{i}w)$ by the assumption, 
the equality $w \le_{s_iw} u$ contradicts Lemma~\ref{lem:tec1}.
This proves the lemma. 
\end{proof}

We set
%
%
\begin{equation} \label{eq:fin}
\fin(\kp{w}):=
\bigl\{ u \in W \mid 
 \text{\rm $ut_{\xi} \in \kp{w}$ for some $\xi \in Q^{\vee}$}
\bigr\}. 
\end{equation}
%
%
\begin{rem} \label{rem:fin}
We deduce by Lemma~\ref{lem:NNS} (with $\J=\emptyset$) that
$u \in \fin(\kp{w})$ if and only if $ut_{\wt(w \Rightarrow u)} \in \kp{w}$. 
Indeed, the ``if'' part is obvious from the definition. 
Let us show the ``only if'' part. 
Let $u \in \fin(\kp{w})$. By the definition, 
there exists $\xi \in Q^{\vee}$ such that $ut_{\xi} \in \kp{w}$. 
Since $ut_{\xi} \in \kp{w} \subset (W_{\af})_{\sige w}$, 
we see from Lemma~\ref{lem:NNS} that $\xi \ge \wt(w \Rightarrow u)$, and 
hence $ut_{\xi} \sige ut_{\wt(w \Rightarrow u)}$. 
If $ut_{\wt(w \Rightarrow u)} \in (W_{\af})_{\sige z}$ for some $z \in W$ 
such that $z > w$, then we have $ut_{\xi} \sige 
ut_{\wt(w \Rightarrow u)} \sige z$, which contradicts 
the fact that $ut_{\xi} \in \kp{w}$.
\end{rem}
%
%
\begin{lem} \label{lem:fin}
For every $w \in W$, we have 
$\fin(\kp{w}) = \EQB(w)$. 
\end{lem}

\begin{proof}
First we prove that $\fin(\kp{w}) \subset \EQB(w)$. 
Let $u \in \fin(\kp{w})$
(note that $ut_{\wt(w \Rightarrow u)} \in \kp{w}$ by Remark~\ref{rem:fin}), 
and suppose, for a contradiction, 
that $u \notin \EQB(w)$. By Lemma~\ref{lem:tec1}, 
there exists $z \in W$ such that $z > w$ and $z \le_{w} u$. 
Since $z \le_{w} u$, we see that $\wt(z \Rightarrow u) \le \wt(w \Rightarrow u)$, 
which implies that $ut_{\wt(w \Rightarrow u)} \sige ut_{\wt(z \Rightarrow u)} \sige z$ 
by Lemma~\ref{lem:NNS} (with $\J=\emptyset$). 
This contradicts the fact that $ut_{\wt(w \Rightarrow u)} \in \kp{w}$. 

Next we prove that $\fin(\kp{w}) \supset \EQB(w)$. 
Let $u \in \EQB(w)$. It suffices to show that 
$ut_{\wt(w \Rightarrow u)} \in \kp{w}$. 
It is obvious from Lemma~\ref{lem:NNS} that 
$ut_{\wt (w \Rightarrow u)} \in (W_{\af})_{\sige w}$. 
Suppose, for a contradiction, that $ut_{\wt (w \Rightarrow u)} \sige z$ 
for some $z \in W$ such that $z > w$; it follows from Lemma~\ref{lem:NNS} that 
$\wt(z \Rightarrow u) \le \wt(w \Rightarrow u)$. 
Take arbitrary shortest directed paths
\begin{equation*}
w = x_{0} \edge{\gamma_{1}} \cdots \edge{\gamma_{a}} x_{a} = z, \qquad
z = x_{a} \edge{\gamma_{a+1}} \cdots \edge{\gamma_{b}} x_{b} = u
\end{equation*}
in $\QB$, and concatenate these as:
\begin{equation*}
\bp : 
w = x_{0} \edge{\gamma_{1}} \cdots \edge{\gamma_{a}} x_{a} = z 
\edge{\gamma_{a+1}} \cdots \edge{\gamma_{b}} x_{b} = u.
\end{equation*}
Since $z > w$ in the (ordinary) Bruhat order on $W$, we deduce that 
all the edges in the shortest directed path above from $w$ to $z$ are Bruhat edges, 
which implies that $\wt(\bp) = \wt (z \Rightarrow u)$. 
Also, it follows from Proposition~\ref{prop:81} 
(with $\J=\emptyset$) that $\wt(\bp) \ge \wt(w \Rightarrow u)$. 
Therefore, we obtain 
\begin{equation*}
\wt(z \Rightarrow u) \le \wt(w \Rightarrow u) \le 
\wt(\bp) = \wt (z \Rightarrow u), 
\end{equation*}
and hence $\wt(w \Rightarrow u) = \wt (\bp)$. 
In particular, we deduce from Proposition~\ref{prop:81}
that $\bp$ is a shortest directed path from $w$ to $u$. 
Hence we obtain $z \le_{w} u$.
Since $u \in \EQB(w)$ and $z > w$ by our assumption, 
the inequality $z \le_{w} u$ contradicts Lemma~\ref{lem:tec1}. 
Thus we have shown that $u \in \fin(\kp{w})$. 
This proves the lemma. 
\end{proof}

\begin{proof}[Proof of Proposition~\ref{prop:kappa}]
For $w \in W$, we set
\begin{equation}
\kp{w}':=
\bigl\{ u t_{\xi} \in W_{\af} \mid 
 u \in \EQB(w),\,\xi \in \wt(w \Rightarrow u) + Q_{I_{w}}^{\vee,+} \bigr\}.
\end{equation}
We prove the equality $\kp{w}=\kp{w}'$ 
by descending induction on $\ell(w)$. 
If $w = \lng$, then the equality follows from Lemma~\ref{lem:w0}. 
Let $w \in W$ and $i \in I$ be such that $s_{i}w < w$. 
Assume that $\kp{w}=\kp{w}'$ (the induction hypothesis). 
We will prove that $\kp{s_{i}w}=\kp{s_{i}w}'$. 

\paragraph{Step 1.}
%
We prove the inclusion $\kp{s_{i}w} \subset \kp{s_{i}w}'$. 
Let $x \in \kp{s_{i}w}$, and write it as 
$x = ut_{\xi} \in \kp{s_{i}w}$ for some $u \in W$ and $\xi \in Q^{\vee}$. 
We see from Lemma~\ref{lem:fin} that $u \in \EQB(s_{i}w)$. 
Hence it remains to show that $\xi \in \wt(s_{i}w \Rightarrow u) + Q^{\vee,+}_{I_{s_iw}}$. 
Since $u \in \EQB(s_{i}w)$ and $s_{i}w < w$, it follows from Lemma~\ref{lem:tec2} that 
$u^{-1}\alpha_{i} \in \Delta^{+}$. 
Therefore, from the assumption that $x = ut_{\xi} \sige s_{i}w$, 
we see by Lemma~\ref{lem:dia-si}\,(3) that
$s_{i}ut_{\xi} \sige w$, and hence 
$s_{i}ut_{\xi} \in (W_{\af})_{\sige w}$. 
Suppose, for a contradiction, that $s_{i}ut_{\xi} \sige z$ for some 
$z \in W$ such that $z > w$. 
If $z^{-1}\alpha_{i} \in \Delta^{+}$, then 
we see by Lemma~\ref{lem:dia-si}\,(1) that $ut_{\xi} \sige z$. 
Hence we have $ut_{\xi} \sige z > w > s_{i}w$, 
which contradicts the assumption that $ut_{\xi} \in \kp{s_{i}w}$. 
Similarly, 
if $z^{-1}\alpha_{i} \in \Delta^{-}$, then 
we see by Lemma~\ref{lem:dia-si}\,(3) that $ut_{\xi} \sige s_{i}z$. 
Hence we have $ut_{\xi} \sige s_{i}z > s_{i}w$ (by Lemma~\ref{lem:dia-b}\,(3)), 
which contradicts the assumption that $ut_{\xi} \in \kp{s_{i}w}$. 
Thus we conclude that $s_{i}ut_{\xi} \in \kp{w}$. 
From this, by the induction hypothesis, we obtain 
$\xi \in \wt(w \Rightarrow s_{i}u) + Q_{I_{w}}^{\vee,+}$; 
note that $\wt(w \Rightarrow s_{i}u) = \wt(s_{i}w \Rightarrow u)$ 
by Lemma~\ref{lem:dia-qb}\,(2). 
%
\paragraph{Case 1.1}
%
Assume that $s_{i}w \notin wW_{I_{w}}$. 
In this case, we have $I_{s_{i}w} = I_{w}$ 
by Lemma~\ref{lem:311}\,(a), and hence 
$\xi \in \wt(s_{i}w \Rightarrow u) + Q_{I_{s_{i}w}}^{\vee,+}$, as desired. 
%
\paragraph{Case 1.2}
%
Assume that $s_{i}w \in wW_{I_{w}}$; by Lemma~\ref{lem:311}\,(b), 
$I_{s_{i}w} = I_{w} \setminus \{k\}$, where $\alpha_{k}=-w^{-1}\alpha_{i}$. 
Suppose that $\xi \in \wt(s_{i}w \Rightarrow u) + 
( Q_{I_{w}}^{\vee,+} \setminus Q_{I_{s_{i}w}}^{\vee,+} )$, namely, that 
the coefficient of $\alpha_{k}^{\vee}$ in 
$\xi - \wt(s_{i}w \Rightarrow u)$ is greater than $0$. 
Since $u \in \EQB(s_{i}w)$, we deduce from Proposition~\ref{prop:323}\,(2) that 
$\wt ( w \Rightarrow u ) = \wt ( s_{i}w \Rightarrow u ) + \alpha_{k}$. 
Hence we see that $\xi - \wt (w \Rightarrow u ) \in Q^{\vee,+}$, 
which implies that $u t_{\xi} \sige w$ by Lemma~\ref{lem:NNS}. 
Since $w > s_{i}w$ by the assumption, 
the inequality $u t_{\xi} \sige w$ contradicts 
the assumption that $x = ut_{\xi} \in \kp{s_{i}w}$. Thus, we have shown that 
$\xi \in \wt(s_{i}w \Rightarrow u) + Q_{I_{s_{i}w}}^{\vee,+}$, as desired.  

\paragraph{Step 2.}
%
We prove the opposite inclusion 
$\kp{s_{i}w} \supset \kp{s_{i}w}'$. 
Let $x=ut_{\xi} \in \kp{s_{i}w}'$, 
where $u \in \EQB(s_{i}w)$ and 
$\xi \in \wt(s_{i}w \Rightarrow u) + Q_{I_{s_iw}}^{\vee,+}$; 
note that $s_{i}u > u$, or equivalently, $u^{-1}\alpha_{i} \in \Delta^{+}$ 
by Lemma~\ref{lem:tec2}. Hence it follows from Lemma~\ref{lem:dia-qb}\,(2) that 
$\wt(s_{i}w \Rightarrow u) = \wt(w \Rightarrow s_{i}u)$. 
Here, by Lemma~\ref{lem:311}, we have $I_{s_iw} \subset I_{w}$, 
and hence $Q_{I_{s_iw}}^{\vee,+} \subset Q_{I_{w}}^{\vee,+}$. 
Therefore, we have $\xi \in \wt(w \Rightarrow s_{i}u) + Q_{I_{w}}^{\vee,+}$. 
Also, we see from Proposition~\ref{prop:323}\,(2) that $s_{i}u \in \EQB(w)$. 
Hence we conclude that 
$s_{i}ut_{\xi} \in \kp{w}$ by the induction hypothesis; namely, 
we have $s_{i}ut_{\xi} \sige w$ and $s_{i}ut_{\xi} \not\sige z$ 
for any $z \in W$ such that $z > w$. 
Since $s_{i}u > u$ and $s_{i}w < w$, 
it follows from Lemma~\ref{lem:dia-si}\,(3) that 
$ut_{\xi} \sige s_{i}w$, i.e., 
$ut_{\xi} \in (W_{\af})_{\sige s_{i}w}$. 
Suppose, for a contradiction, that 
$ut_{\xi} \in (W_{\af})_{\sige z}$ for some $z \in W$ 
such that $z > s_{i}w$. 
Assume first that $s_{i}z > z$, or equivalently, 
$z^{-1}\alpha_{i} \in \Delta^{+}$. 
We see by Lemmas~\ref{lem:dia-si}\,(3) and \ref{lem:dia-b}\,(3) that
$s_{i}ut_{\xi} \sige s_{i}z > w$, which contradicts 
the fact that $s_{i}ut_{\xi} \in \kp{w}$. 
Assume next that $s_{i}z < z$, or equivalently, 
$z^{-1}\alpha_{i} \in \Delta^{-}$. 
We see by Lemma~\ref{lem:dia-b}\,(1) that $z > w$. 
Also, since $s_{i}u > u$, 
we have $s_{i}ut_{\xi} \sige ut_{\xi}$ (by the definition of $\sige$). 
Combining these inequalities, we obtain 
$s_{i}ut_{\xi} \sige ut_{\xi} \sige z > w$, which contradicts 
the fact that $s_{i}ut_{\xi} \in \kp{w}$. 
Thus we conclude that 
$ut_{\xi} \in \kp{s_{i}w}$. 

This completes the proof of Proposition~\ref{prop:kappa}. 
\end{proof}
%
%
\begin{prop} \label{prop:kappa2}
Let $\J$ be a subset of $I$, and let $w \in \WJu${\rm;}
recall from Remark~{\rm \ref{rem:Iw1}} that $\J \subset \Iw$. 
Then, the subset 
\begin{equation*}
\kq{w} = ( (\WJu)_{\af} )_{\sige w} \setminus 
\bigcup_{z \in \WJu,\,z > w} ( (\WJu)_{\af} )_{\sige z}, 
\end{equation*}
defined in \eqref{eq:kw}, is identical to the set
%
%
\begin{equation} \label{eq:k2-1}
\PJ(\kw) = 
\bigl\{ u \PJ(t_{\xi}) \in (\WJu)_{\af} \mid u \in \mcr{\EQB(\xw)}, 
\xi \in \wt^{\J}(w \Rightarrow u) + Q^{\vee,+}_{\Iw \setminus \J}\bigr\}. 
\end{equation}
\end{prop}
\begin{rem}
The equality in \eqref{eq:k2-1} follows from 
Proposition~\ref{prop:kappa}, Lemma~\ref{lem:wtS}, 
and Lemma~\ref{lem:PiJ}\,(1), (3). 
\end{rem}

\begin{rem} \label{rem:kp2}
Keep the setting of Proposition~\ref{prop:kappa2}. 
We see by Lemma~\ref{lem:PiJ}\,(3) that the map 
\begin{equation*}
\mcr{\EQB(\xcr{w})} \times Q^{\vee,+}_{\Iw \setminus \J} \rightarrow \kq{w}, \qquad
(u,\gamma) \mapsto u \PJ(t_{\wt^{\J}(w \Rightarrow u)+\gamma}),
\end{equation*}
is bijective. 
\end{rem}

\begin{proof}[Proof of Proposition~\ref{prop:kappa2}]
First, we claim that for arbitrary $u \in \WJu$ and $u_{1} \in u\WJ$, 
%
%
\begin{equation} \label{eq:k2}
\PJ((W_{\af})_{\sige u_{1}}) = ( (\WJu)_{\af} )_{\sige u}.
\end{equation}
The inclusion $\subset$ follows from Lemma~\ref{lem:INS} and \eqref{eq:PiJ2}. 
Let us prove the opposite inclusion $\supset$. 
Let $x \in ( (\WJu)_{\af} )_{\sige u}$, and write it as 
$x = v \PJ(t_{\xi})$ for some $v \in \WJu$ and $\xi \in Q^{\vee}$.
Since $x \sige u$ by the assumption, we see from Lemma~\ref{lem:NNS} that 
$[\xi]^{\J} \ge \wt^{\J}(u \Rightarrow v)$; note that 
$\wt^{\J}(u \Rightarrow v) = [\wt(u_{1} \Rightarrow v)]^{\J}$
by Lemma~\ref{lem:wtS}. 
Hence we can take $\zeta \in Q^{\vee}$ such that $[\zeta]^{\J}=[\xi]^{\J}$ and 
$\zeta \ge \wt(u_{1} \Rightarrow v)$. We set $y:=vt_{\zeta} \in W_{\af}$. 
Then it follows from Lemma~\ref{lem:NNS} (with $\J=\emptyset$) that $y \sige u_{1}$. 
Also, we have $\PJ(y) = \PJ(v)\PJ(t_{\zeta})=v\PJ(t_{\xi}) = x$ 
by Lemma~\ref{lem:PiJ}\,(1) and (3). 
This proves the opposite inclusion $\supset$, and hence \eqref{eq:k2}. 

From \eqref{eq:k2} and the definitions of $\kw$ and $\kq{w}$ 
(by a standard set-theoretical argument), we see easily 
that $\PJ(\kw) \supset \kq{w}$. 
Let us prove the opposite inclusion $\subset$. 
Let $x \in \kw$; by \eqref{eq:k2}, we have 
$\PJ(x) \in ( (\WJu)_{\af} )_{\sige w}$. 
Suppose, for a contradiction, that 
$\PJ(x) \in ( (\WJu)_{\af} )_{\sige z}$ for some 
$z \in \WJu$ such that $z > w$. 
By Proposition~\ref{prop:kappa}, 
we can write the $x$ as $x = ut_{\xi}$ 
for some $u \in \EQB(\xcr{w})$ and $\xi \in 
\wt(\xcr{w} \Rightarrow u)+\Qvw$. 
Hence we have $\PJ(x) = \mcr{u} \PJ(t_{\xi}) = 
u_{1}t_{\xi+\xi_{1}}$ for some $u_{1} \in u\WJ$ and 
$\xi_{1} \in \QJ^{\vee}$ by Lemma~\ref{lem:PiJ}\,(1) and (2). 
Here, by \cite[Lemma~29]{NNS2}, 
the subset $\EQB(\xcr{w})$ of $W$ is a disjoint union of 
some cosets in $W/W_{\Iw}$; since $\J \subset \Iw$, 
the subset $\EQB(\xcr{w})$ of $W$ is also a disjoint union of some cosets in $W/\WJ$. 
Hence it follows that $u_{1}$ is contained in $\EQB(\xcr{w})$. 
Also, since $\PJ(x) \sige z$ by our assumption, 
it follows from Lemma~\ref{lem:NNS} that 
$[\xi]^{\J} \ge \wt^{\J}(z \Rightarrow \mcr{u})$; 
by Lemma~\ref{lem:wtS}, we have 
$\wt^{\J}(z \Rightarrow \mcr{u}) = [\wt(\xcr{z} \Rightarrow u_{1})]^{\J}$, 
and hence $[\xi]^{\J} \ge [\wt(\xcr{z} \Rightarrow u_{1})]^{\J}$. 
We set
\begin{equation*}
\zeta:=\wt(\xcr{w} \Rightarrow u_{1}) + 
\underbrace{\xi - \wt(\xcr{w} \Rightarrow u)}_{\in \Qvw}
= \xi +
\underbrace{\wt(\xcr{w} \Rightarrow u_{1})
- \wt(\xcr{w} \Rightarrow u)}_{\text{$\in \QJv$ by Lemma~\ref{lem:wtS}}}. 
\end{equation*}
Since $[\xi]^{\J} \ge [\wt(\xcr{z} \Rightarrow u_{1})]^{\J}$, as seen above, 
we can take $\zeta_{1} \in \QJvp$ such that 
$\zeta + \zeta_{1} \ge \wt(\xcr{z} \Rightarrow u_{1})$; 
since $\J \subset \Iw$, we have 
$\zeta + \zeta_{1} \in \wt(\xcr{w} \Rightarrow u_{1}) + \Qvw$. 
Hence it follows from Proposition~\ref{prop:kappa} 
that $y:=u_{1}t_{\zeta+\zeta_{1}} \in \kw$. 
However, since $\zeta + \zeta_{1} \ge \wt(\xcr{z} \Rightarrow u_{1})$, 
we deduce by Lemma~\ref{lem:NNS} (with $S = \emptyset$) that 
$y = u_{1}t_{\zeta+\zeta_{1}} \sige \xcr{z} > \xcr{w}$, 
which contradicts the fact that $y \in \kw$. 
This proves the opposite inclusion $\PJ(\kw) \subset \kq{w}$, 
and hence completes the proof of Proposition~\ref{prop:kappa2}. 
\end{proof}
%
%
\section{Proof of Theorem~\ref{thm:main}.}
\label{sec:prf-main}
%
%
\subsection{Quantum Lakshmibai-Seshadri paths and the degree function.}
\label{subsec:QLS}
We fix $\lambda \in P^{+}$, and take $\J=\J_{\lambda}$ as in \eqref{eq:J}. 
%
%
\begin{dfn} \label{dfn:QBa}
For a rational number $0 < a < 1$, 
we define $\QBa$ to be the subgraph of $\QBJ$ 
with the same vertex set but having only those edges of the form
$u \edge{\beta} v$ for which 
$a\pair{\lambda}{\beta^{\vee}} \in \BZ$ holds.
\end{dfn}
%
%
\begin{dfn}\label{dfn:QLS}
A quantum LS path of shape $\lambda $ is a pair 
%
%
\begin{equation} \label{eq:QLS}
\eta = (\bw \,;\, \ba) 
     = (w_{1},\,\dots,\,w_{s} \,;\, a_{0},\,a_{1},\,\dots,\,a_{s}), \quad s \ge 1, 
\end{equation}
of a sequence $\bw : w_{1},\,\dots,\,w_{s}$ 
of elements in $\WJu$ with $w_{u} \ne w_{u+1}$ 
for any $1 \le u \le s-1$ and an increasing sequence 
$\ba : 0 = a_0 < a_1 < \cdots  < a_s =1$ of rational numbers 
satisfying the condition that there exists a directed path 
from $w_{u+1}$ to  $w_{u}$ in $\QBb{a_{u}}$ 
for each $u = 1,\,2,\,\dots,\,s-1$. 
\end{dfn}

We denote by $\QLS(\lambda)$ 
the set of all quantum LS paths of shape $\lambda$.
If $\eta \in \QLS(\lambda)$ is of the form \eqref{eq:QLS}, 
then we set $\kappa(\eta):=w_{s} \in \WJu$, 
and call it the final direction of $\eta$. 
For $w \in W$, we set (see \cite[Sect.~3.2]{NNS1} and \cite[Sect.~2.3]{NNS2})
%
%
\begin{equation} \label{eq:QLSw}
\QLS^{w\lambda,\infty}(\lambda) := 
 \bigl\{
   \eta \in \QLS(\lambda) \mid \kappa(\eta) \in \mcr{\EQB{\xw}}
 \bigr\}. 
\end{equation}

We define a projection $\cl : (\WJu)_{\af} \twoheadrightarrow \WJu$ by
$\cl (x) := w$ for $x \in (\WJu)_{\af}$ of the form 
$x = w\PJ(t_{\xi})$ with $w \in \WJu$ and $\xi \in Q^{\vee}$.
For $\pi = (x_{1},\,\dots,\,x_{s}\,;\,a_{0},\,a_{1},\,\dots,\,a_{s}) 
\in \SLS(\lambda)$, we define 
\begin{equation*}
\cl(\pi):=(\cl(x_{1}),\,\dots,\,\cl(x_{s})\,;\,a_{0},\,a_{1},\,\dots,\,a_{s});
\end{equation*}
here, for each $1 \le p < q \le s$ such that $\cl(x_{p})= \cdots = \cl(x_{q})$, 
we drop $\cl(x_{p}),\,\dots,\,\cl(x_{q-1})$ and $a_{p},\,\dots,\,a_{q-1}$. 
We know from \cite[Sect.~6.2]{NS16} that $\cl(\pi) \in \QLS(\lambda)$ 
for all $\pi \in \SLS(\lambda)$, and that 
the map $\cl:\SLS(\lambda) \rightarrow \QLS(\lambda)$, 
$\pi \mapsto \cl(\pi)$, is surjective. 
We also know the following lemma from \cite[Lemma~6.2.3]{NS16}; 
recall that $\SLS_{0}(\lambda)$ denotes the connected component of $\SLS(\lambda)$ 
containing $\pi_{\lambda}=(e\,;\,0,1)$. 
%
%
\begin{lem} \label{lem:deg}
For each $\eta \in \QLS(\lambda)$, there exists a unique 
$\pi_{\eta} \in \SLS_{0}(\lambda)$ such that 
$\cl(\pi_{\eta})=\eta$ and $\kappa(\pi_{\eta}) = \kappa(\eta) \in \WJu$. 
\end{lem}

Now we define the (tail) degree function $\deg_{\lambda}:
\QLS(\lambda) \rightarrow \BZ_{\le 0}$ as follows. 
Let $\eta \in \QLS(\lambda)$, and 
take $\pi_{\eta} \in \SLS_{0}(\lambda)$ as in Lemma~\ref{lem:deg}. 
We see from the argument in \cite[Sect.~6.2]{NS16} that 
$\wt(\pi_{\eta}) = \lambda - \gamma + k\delta$ 
for some $\gamma \in Q^{+}$ and $k \in \BZ_{\le 0}$. 
Then we set $\deg_{\lambda}(\eta):=k$. 
We know the following description of $\deg_{\lambda}$ 
from \cite[Corollary~4.8]{LNSSS2}; for the definition of 
$\wt^{\J}(u \Rightarrow v)$, see Section~\ref{subsec:QBG}. 
%
%
\begin{prop} \label{prop:deg}
For $\eta = 
(w_{1},\,\dots,\,w_{s} \,;\, a_{0},\,a_{1},\,\dots,\,a_{s}) \in \QLS(\lambda)$, 
it holds that
%
%
\begin{equation} \label{eq:deg0}
\deg_{\lambda}(\eta) = - \sum_{u=1}^{s-1} 
a_{u} \pair{\lambda}{\wt^{\J}(w_{u+1} \Rightarrow w_{u})}.
\end{equation}
\end{prop}

Also, for $\eta = 
(w_{1},\,\dots,\,w_{s} \,;\, a_{0},\,a_{1},\,\dots,\,a_{s}) \in \QLS(\lambda)$ and 
$w \in \WJu$, we define the degree of $\eta$ at $w\lambda$ 
(see \cite[Sect.~3.2]{NNS1} and \cite[Sect.~2.3]{NNS2}) by 
%
%
\begin{equation} \label{eq:degw}
\deg_{w\lambda}(\eta):=
- \sum_{u=1}^{s} 
a_{u} \pair{\lambda}{\wt^{\J}(w_{u+1} \Rightarrow w_{u})}, \quad 
\text{with $w_{s+1}:=w$}.
\end{equation}
%
%
\begin{thm}[{\cite[Theorem~3.2.7]{NNS1}}] \label{thm:NNS}
Let $\lambda \in P^{+}$, and take $\J=\J_{\lambda}$ as in \eqref{eq:J}. 
Then, 
\begin{equation}
\sum_{\eta \in \QLS^{w\lambda,\infty}(\lambda)} 
\be^{\wt(\eta)}q^{\deg_{w\lambda}(\eta)} =
E_{w\lambda}(q,\infty) \qquad 
\text{\rm for $w \in \WJu$}. 
\end{equation}
\end{thm}

%
\subsection{Proof of the graded character formula for $\Kg_{w}^{-}(\lambda)$.}
\label{subsec:prf-main1}

Let $\lambda=\sum_{i \in I}m_{i}\vpi_{i} \in P^{+}$, 
and take $\J=\J_{\lambda}$ as in  \eqref{eq:J}. 
Let $w \in \WJu$. Recall from \eqref{eq:gch2} that 
\begin{equation*}
\gch \Kg_{w}^{-}(\lambda) = 
\sum_{\pi \in \Kc{w}(\lambda)} \be^{\fwt(\wt(\pi))} q^{\qwt(\wt(\pi))}, 
\end{equation*}
where $\Kc{w}(\lambda) = \bigl\{ \pi \in \SLS(\lambda) 
\mid \kappa(\pi) \in \kq{w} \bigr\}$ by \eqref{eq:BK2}. Because 
$\Kc{w}(\lambda) = 
\bigsqcup_{\eta \in \QLS(\lambda)} 
 \bigl( \cl^{-1}(\eta) \cap 
 \Kc{w}(\lambda) \bigr)$, 
we deduce that
\begin{equation} \label{eq:ch1}
\gch \Kg_{w}^{-}(\lambda) = 
\sum_{\eta \in \QLS(\lambda)} 
\Biggl(
 \underbrace{
   \sum_{\pi \in \cl^{-1}(\eta) \cap \Kc{w}(\lambda)} 
   \be^{\fwt(\wt(\pi))} q^{\qwt(\wt(\pi))}
 }_{=:F_{\eta}}
 \Biggr).
\end{equation}
Let us compute $F_{\eta}$ for each $\eta \in \QLS(\lambda)$. 
We fix an arbitrary $\eta \in \QLS(\lambda)$, and take 
$\pi_{\eta} \in \SLS_{0}(\lambda)$ as in Lemma~\ref{lem:deg}. 
We take and fix a monomial $X_{\eta}$ in root operators $e_{i}$ and $f_{i}$, $i \in I_{\af}$, 
such that $\pi_{\eta}=X_{\eta}\pi_{\lambda}$; we see by \cite[Lemma~6.2.2]{NS16} that
%
%
\begin{equation} \label{eq:cl-inv}
\cl^{-1}(\eta)=\bigl\{X_{\eta} (t_{\xi} \cdot \pi^{C}) \in \SLS(\lambda) \mid 
 C \in \Conn(\SLS(\lambda)),\,
 \xi \in Q_{\Jc}^{\vee} \bigr\}; 
\end{equation}
for the definitions of $\pi^{C} \in \SLS(\lambda)$ and 
$x \cdot \pi$ with $x \in W_{\af}$ and $\pi \in \SLS(\lambda)$, 
see \eqref{eq:etaC} and Remark~\ref{rem:weyl}, respectively. 
By the argument after \cite[(5.4)]{NNS1}, we see that 
\begin{equation} \label{eq:ke}
\kappa(X_{\eta} (t_{\xi} \cdot \pi^{C})) = \kappa(\eta)\PJ(t_{\xi})
\quad \text{for $C \in \Conn(\SLS(\lambda))$ and $\xi \in Q_{\Jc}^{\vee}$}.
\end{equation}
Because 
\begin{equation*}
\kq{w}=
\bigl\{ u \PJ(t_{\xi}) \in (\WJu)_{\af} \mid u \in \mcr{\EQB(\xw)}, 
\xi \in \wt^{\J}(w \Rightarrow u) + Q^{\vee,+}_{\Iw \setminus \J}\bigr\}
\end{equation*}
by Proposition~\ref{prop:kappa2}, we deduce the following: 
\begin{enu}
\item[(i)] $\cl^{-1}(\eta) \cap \Kc{w}(\lambda) \ne \emptyset$ 
$\iff$ $\kappa(\eta) \in \mcr{\EQB(\xw)}$ $\iff$
$\eta \in \QLS^{w\lambda,\infty}(\lambda)$; 
the implication $\Leftarrow$ in the first equivalence follows 
from \eqref{eq:ke} and the fact that 
$\wt^{\J}(w \Rightarrow u) + Q^{\vee,+}_{\Iw \setminus \J} \subset Q_{\Jc}^{\vee}$.

\item[(ii)] If $\eta \in \QLS^{w\lambda,\infty}(\lambda)$, then
%
%
\begin{equation} \label{eq:iia}
\begin{split}
& \cl^{-1}(\eta) \cap \Kc{w}(\lambda) \\
& = 
\left\{ X_{\eta} (t_{\xi} \cdot \pi^{C}) \in \SLS(\lambda) \ \Biggm| \ 
 \begin{array}{l}
 C \in \Conn(\SLS(\lambda)), \\[1mm]
 \xi \in \wt^{\J}(w \Rightarrow \kappa(\eta)) + Q^{\vee,+}_{\Iw \setminus \J}
 \end{array} \right\}; 
\end{split}
\end{equation}
remark that for $C,\,C' \in \Conn(\SLS(\lambda))$ and 
$\xi,\,\xi'\in \wt^{\J}(w \Rightarrow \kappa(\eta)) + Q^{\vee,+}_{\Iw \setminus \J}$, 
%
%
\begin{equation} \label{eq:iib}
X_{\eta}(t_{\xi} \cdot \pi^{C}) = X_{\eta} (t_{\xi'} \cdot \pi^{C'}) \iff
\text{$C=C'$ and $\xi = \xi'$}. 
\end{equation}
Indeed, the implication $\Leftarrow$ is obvious. 
Let us show the implication $\Rightarrow$. 
Since $X_{\eta}(t_{\xi} \cdot \pi^{C}) \in C$ and 
$X_{\eta} (t_{\xi'} \cdot \pi^{C'}) \in C'$, 
we have $C=C'$. Also, since 
$\kappa(\eta)\PJ(t_{\xi}) = 
\kappa(X_{\eta}(t_{\xi} \cdot \pi^{C})) = 
\kappa(X_{\eta}(t_{\xi'} \cdot \pi^{C'})) = 
\kappa(\eta)\PJ(t_{\xi'})$, as seen above, 
we deduce that $\xi - \xi' \in \QJv$ by Lemma~\ref{lem:PiJ}\,(3). 
Therefore, by the assumption that 
$\xi,\,\xi'\in \wt^{\J}(w \Rightarrow \kappa(\eta)) + 
Q^{\vee,+}_{\Iw \setminus \J}$, we obtain $\xi=\xi'$.
\end{enu}

Let $C \in \Conn(\SLS(\lambda))$, and write $\Theta(C) \in \Par(\lambda)$ as
$\Theta(C)=(\rho^{(i)})_{i \in I}$, where
$\rho^{(i)} = (\rho^{(i)}_{1} \ge \cdots \ge \rho^{(i)}_{m_{i}-1} \ge \rho^{(i)}_{m_{i}}=0)$ 
for each $i \in \Jc$, and $\rho^{(i)}=\emptyset$ for all $i \in \J$.  
Also, let $\xi \in \wt^{\J}(w \Rightarrow \kappa(\eta)) + 
Q^{\vee,+}_{\Iw \setminus \J}$, and write it as: 
\begin{equation*}
\xi = \wt^{\J}(w \Rightarrow \kappa(\eta)) + 
\sum_{i \in \Iw \setminus \J} c_{i}\alpha_{i}^{\vee},
\end{equation*}
where $c_{i} \in \BZ_{\ge 0}$ for $i \in \Iw \setminus \J$; 
by convention, we set $c_{i}:=0$ for all $i \in \J$. 
For each $i \in \Iw$, we set $\ti{\rho}^{(i)} := 
(c_{i}+\rho^{(i)}_{1} \ge \cdots \ge c_{i}+\rho^{(i)}_{m_{i}-1} \ge 
 c_{i}+\rho^{(i)}_{m_{i}-1} = c_{i})$, 
which is a partition of length less than or equal to $m_{i}$; 
note that $\ti{\rho}^{(i)} = \emptyset$ for all $i \in S$. 
Also, for each $i \in I \setminus \Iw$, 
we set $\ti{\rho}^{(i)}:=\rho^{(i)}$, 
which is a partition of length less than $m_{i}$. 
Then, $(\ti{\rho}^{(i)})_{i \in I}$ is an element of 
the set $\Par_{w}(\lambda)$ of $I$-tuples 
$\bchi = (\chi^{(i)})_{i \in I}$ of partitions 
such that for each $i \in \Iw$ (resp., $i \in I \setminus \Iw$), 
$\chi^{(i)}$ is a partition of length less than or 
equal to (resp., less than) $m_{i}$. By the same calculations as 
those after \cite[(6.3.3)]{NS16} and \cite[(5.5)]{NNS1}, 
we deduce that
%
%
\begin{equation} \label{eq:wtXS}
\wt ( X (t_{\xi} \cdot \pi^{C}) )
  = \wt (\eta) + \bigl( 
    \underbrace{
    \deg_{\lambda}(\eta) - 
    \pair{\lambda}{\wt^{\J}(w \Rightarrow \kappa(\eta))}}_{%
    \text{$= \deg_{w\lambda}(\eta)$ by \eqref{eq:deg0} and \eqref{eq:degw}}%
    } \bigr) \delta
  - |(\ti{\rho}^{(i)})_{i \in I}| \delta. 
\end{equation}
Summarizing the above, 
for each $\eta \in \QLS^{w\lambda,\infty}(\lambda)$,
\begin{align*}
F_{\eta} & = 
\sum_{\pi \in \cl^{-1}(\eta) \cap \Kc{w}(\lambda)} 
   \be^{\fwt(\wt(\pi))} q^{\qwt(\wt(\pi))} \\[3mm]
& =
  \sum_{
      \begin{subarray}{c} 
       C \in \Conn(\SLS(\lambda)) \\[1mm] 
       \xi \in \wt^{\J}(w \Rightarrow \kappa(\eta)) + Q^{\vee,+}_{\Iw \setminus \J}
      \end{subarray}}
  \be^{\fwt(\wt ( X_{\eta} (t_{\zeta} \cdot \pi^{C}) ))} 
  q^{\qwt(\wt ( X_{\eta} (t_{\zeta} \cdot \pi^{C}) ))} 
\quad \text{by \eqref{eq:iia} and \eqref{eq:iib}} \\[3mm]
& = \be^{\wt(\eta)} q^{\deg_{w\lambda}(\eta)} 
     \sum_{\bchi \in \Par_{w}(\lambda)} q^{-|\bchi|} 
\quad \text{by \eqref{eq:wtXS}},
\end{align*}
where $|\bchi|=\sum_{i \in I} |\chi^{(i)}|$
for $\bchi = (\chi^{(i)})_{i \in I} \in \Par_{w}(\lambda)$ 
(see Section~\ref{subsec:Parp}). Here we recall that 
if $\bchi = (\chi^{(i)})_{i \in I} \in \Par_{w}(\lambda)$, 
then the partition $\chi^{(i)}$ is a partition of 
length less than or equal to $m_{i}-\eps_{i}$ for each $i \in I$, 
where $\eps_{i}$ is as in \eqref{eq:eps}.
Therefore, we deduce that
\begin{equation*}
F_{\eta}
 = \be^{\wt(\eta)} q^{\deg_{w\lambda}(\eta)} 
     \sum_{\bchi \in \Par_{w}(\lambda)} q^{-|\bchi|}
 = \be^{\wt(\eta)} q^{\deg_{w\lambda}(\eta)} 
   \left( \prod_{i \in I} \prod_{r=1}^{m_{i}-\eps_{i}}(1-q^{-r})\right)^{-1}.
\end{equation*}
Also, we have $F_{\eta}=0$ for all 
$\eta \in \QLS(\lambda) \setminus \QLS^{w\lambda,\infty}(\lambda)$, 
since $\cl^{-1} \cap \Kc{w}(\lambda) = \emptyset$, 
as seen in (i) above. Substituting these $F_{\eta}$'s into \eqref{eq:ch1}, 
we obtain 
\begin{align*}
\gch \Kg_{w}^{-}(\lambda) & = 
\sum_{\eta \in \QLS^{w\lambda,\infty}(\lambda)} 
\be^{\wt(\eta)} q^{\deg_{w\lambda}(\eta)} 
    \left( \prod_{i \in I} \prod_{r=1}^{m_{i}-\eps_{i}}(1-q^{-r})\right)^{-1} \\[3mm]
& = \left( \prod_{i \in I} \prod_{r=1}^{m_{i}-\eps_{i}}(1-q^{-r})\right)^{-1}
E_{w\lambda}(q,\infty) \quad \text{by Theorem~\ref{thm:NNS}}. 
\end{align*}
This completes the proof of equation \eqref{eq:main1}.
%
%
\subsection{Proof of the graded character formula for $\Kl_{w}^{-}(\lambda)$.}
\label{subsec:prf-main2}

Let $\lambda \in P^{+}$, and take $\J=\J_{\lambda}$ as in \eqref{eq:J}. Let $w \in \WJu$. 
For each $\eta \in \QLS(\lambda)$, we take and fix a monomial $X_{\eta}$ in 
root operators $e_{i}$ and $f_{i}$, $i \in I_{\af}$, such that 
$\pi_{\eta}=X_{\eta}\pi_{\lambda}$, where we take $\pi_{\eta}$ as in Lemma~\ref{lem:deg}. 
We see from \cite[Theorem~5.12]{NNS1} that 
under the isomorphism $\Psi_{\lambda}:\CB(\lambda) \stackrel{\sim}{\rightarrow} 
\SLS(\lambda)$ of crystals (see Theorem~\ref{thm:isom}), 
the subset $\CB(X_{w}^{-}(\lambda))$ of $\CB(\lambda)$ (see \eqref{eq:cb2}) 
is mapped to 
\begin{equation}
\BX_{w}^{-}(\lambda):=
\SLS_{\sige w}(\lambda) \setminus 
\bigl\{ X_{\eta}(t_{\wt^{\J}(w \Rightarrow \kappa(\eta))} \cdot \pi_{\lambda}) 
 \mid \eta \in \QLS(\lambda) \bigr\}. 
\end{equation}
Therefore, we have 
\begin{equation} \label{eq:main2a}
\gch \Kl_{w}^{-}(\lambda) = 
\sum_{\pi \in \Kc{w}(\lambda) \setminus \BX_{w}^{-}(\lambda)} 
\be^{\fwt(\wt(\pi))} q^{\qwt(\wt(\pi))}. 
\end{equation}
Here, by (i) and (ii) in the previous subsection, we see that 
\begin{equation*}
\Kc{w}(\lambda) 
= \left\{ X_{\eta} (t_{\xi} \cdot \pi^{C}) \in \SLS(\lambda) \ \biggm| \ 
\begin{array}{l}
C \in \Conn(\SLS(\lambda)),\,\eta \in \QLS^{w\lambda,\infty}(\lambda), \\[1mm]
\xi \in \wt^{\J}(w \Rightarrow \kappa(\eta)) + Q^{\vee,+}_{\Iw \setminus \J}
\end{array}
\right\}.
\end{equation*}
From this, we obtain
\begin{equation} \label{eq:main2b}
\Kc{w}(\lambda) \setminus \BX_{w}^{-}(\lambda) = \bigl\{ 
 X_{\eta}(t_{\wt^{\J}(w \Rightarrow \kappa(\eta))} \cdot \pi_{\lambda}) 
 \mid \eta \in \QLS^{w\lambda,\infty}(\lambda)
\bigr\}.
\end{equation}
Combining \eqref{eq:main2a} and \eqref{eq:main2b}, 
we conclude by the same computation as in the previous subsection 
(this time, only the term corresponding to 
$\bchi = (\emptyset)_{i \in I} \in \Par_{w}(\lambda)$ remains) that
\begin{equation*}
\gch \Kl_{w}^{-}(\lambda) = 
\sum_{\eta \in \QLS^{w\lambda,\infty}(\lambda)} 
\be^{\wt(\eta)} q^{\deg_{w\lambda}(\eta)} = 
E_{w\lambda}(q,\infty). 
\end{equation*}
This completes the proof of equation \eqref{eq:main2}.

\appendix

\section*{Appendix.}
%
%
\section{Recursive proof of the formula \eqref{eq:main1} for $\Kg_{w}^{-}(\lambda)$.}
\label{sec:prf2}
%
%
\begin{prop} \label{prop:interval}
Let $\J$ be a subset of $I$. 
For each $x \in (\WJu)_{\af}$ of the form $x=w\PJ(t_{\xi})$ with 
$w \in \WJu$ and $\xi \in Q^{\vee,+}=\sum_{i \in I} \BZ_{\ge 0} \alpha_{i}^{\vee}$, 
the subset $\bigl\{ v \in \WJu \mid v \sile x \bigr\}$ of $\WJu$ 
has a unique maximal element in the Bruhat order $\le$ on $W$, 
which we denote by $v_{x}$. 
\end{prop}

\begin{proof}
By Lemma~\ref{lem:INS} and \cite[Proposition~2.5.1]{BB}, 
it suffices to prove the assertion in the case that $\J = \emptyset$. 

We take and fix $\gamma \in Q^{\vee}$ such that 
$\pair{\gamma}{\alpha_{i}^{\vee}} > 0$ for all $i \in I$. 
Since $W \cup \{ x \}$ is a finite set, 
we deduce from \cite[Claim~4.14]{Soe97} 
(see also \cite[Lectures 12, 13]{Pet97} and \cite[Appendix~A.3]{INS}) 
that there exists $N \in \BZ_{\ge 0}$ such that for $y,\,z \in W \cup \{ x \}$, 
\begin{enu}
\item[(a)] if $y \sile z$, then 
$yt_{n\gamma} \le zt_{n\gamma}$ for all $n \ge N$
in the (ordinary) Bruhat order $\ge$ on $W_{\af}$; 

\item[(b)] if $yt_{n\gamma} \le zt_{n\gamma}$ for some $n \ge N$, 
then $y \sile z$.
\end{enu}
Let $n \ge N$. Because $W$ is a parabolic subgroup of $W_{\af}$, 
it follows from \cite[Lemma 11\,(ii)]{LLM} 
(note that $t_{n\gamma} \le xt_{n\gamma}$ by (a) 
since $n \ge N$ and $e \sile x$) that 
the subset $\bigl\{ vt_{n\gamma} \mid v \in W,\,
vt_{n\gamma} \le x t_{n\gamma} \bigr\}$ of $W_{\af}$ 
has a unique maximal element 
in the Bruhat order on $W_{\af}$, 
which we write as $v_{n}t_{n\gamma}$ with $v_{n} \in W$. 
Here we see by (a) and (b) that $v_{n}$ does not depend on $n \ge N$, 
and is a unique maximal element of 
$\bigl\{ v \in W \mid v \sile x \bigr\}$ 
in the semi-infinite Bruhat order $\sile$ on $W_{\af}$ 
(and hence in the Bruhat order on $W$; see Remark~\ref{rem:SB}).
This proves the proposition. 
\end{proof}

Let $\J$ be a subset of $I$. 
For $w \in \WJu$, we set $\WJu_{\ge w}:=
\bigl\{ v \in \WJu \mid v \ge w \bigr\}$. 
Also, recall from \eqref{eq:kw} 
the definition of $\kq{v}$ for $v \in \WJu$. 
%
%
\begin{lem} \label{lem:dj2}
Keep the notation and setting above. 
It holds that 
%
%
\begin{equation} \label{eq:dj2}
( (\WJu)_{\af} )_{\sige w} = 
\bigsqcup_{v \in \WJu_{\ge w}} \kq{v}. 
\end{equation}
\end{lem}

\begin{proof}
First we prove the equality: 
\begin{equation} \label{eq:dj1z}
( (\WJu)_{\af} )_{\sige w} = 
\bigcup_{v \in \WJu_{\ge w}} \kq{v}. 
\end{equation}
Let $x \in ( (\WJu)_{\af} )_{\sige w}$; 
note that $w \in \bigl\{v \in \WJu \mid v \sile x \bigr\}$. 
By Proposition~\ref{prop:interval}, 
the subset $\bigl\{v \in \WJu \mid v \sile x \bigr\}$ of $\WJu$ 
has a unique maximal element $v_{x}$. 
Hence it follows from the definition of $v_{x}$ that
$w \le v_{x} \sile x$; in particular, 
$x \in ( (\WJu)_{\af} )_{\sige v_{x}}$. 
If $x \notin \kq{v_{x}}$, then 
$x \in ( (\WJu)_{\af} )_{\sige u}$ for some $u \in \WJu$ 
such that $u > v_{x}$. Since $u \in \bigl\{v \in \WJu \mid v \sile x \bigr\}$, 
this contradicts the maximality of $v_{x}$. 
Hence, $x \in \kq{v_{x}} \subset \bigcup_{v \in \WJu_{\ge w}} \kq{v}$. 
Thus we have shown the inclusion $\subset$. 
Let us show the opposite inclusion $\supset$. 
Let $x \in \bigcup_{v \in \WJu_{\ge w}} \kq{v}$, 
and take $v \in \WJu_{\ge w}$ such that $x \in \kq{v}$. 
Then, by the definition of $\kq{v}$, 
we have $x \in ( (\WJu)_{\af} )_{\sige v}$. 
Since $v \ge w$, we have 
$( (\WJu)_{\af} )_{\sige v} \subset ( (\WJu)_{\af} )_{\sige w}$, 
and hence $x \in ( (\WJu)_{\af} )_{\sige w}$. 
Thus we have shown the inclusion $\supset$, as desired. 

The remaining task is to prove that 
$\kq{v_{1}} \cap \kq{v_{2}} = \emptyset$ 
for all $v_{1},\,v_{2} \in \WJu$ with $v_{1} \ne v_{2}$. 
Suppose, for a contradiction, that 
$\kq{v_{1}} \cap \kq{v_{2}} \ne \emptyset$
for some $v_{1},\,v_{2} \in \WJu$ with $v_{1} \ne v_{2}$, 
and let $x \in \kq{v_{1}} \cap \kq{v_{2}}$; 
note that $x$ is of the form $x=w\PJ(t_{\xi})$ for some 
$w \in \WJu$ and $\xi \in Q^{\vee,+}$ 
since $x \sige v_{1} \in \WJu$ (see Lemmas~\ref{lem:NNS} and 
\ref{lem:PiJ}\,(3)). Therefore, by Proposition~\ref{prop:interval}, 
the subset $\bigl\{ v \in \WJu \mid v \sile x \bigr\}$ of $\WJu$ 
has a unique maximal element $v_{x}$ in the Bruhat order $\le$ on $W$. 
Since $x \in \kq{v_{1}} \subset ( (\WJu)_{\af} )_{\sige v_{1}}$, 
we have $v_{1} \in \bigl\{ v \in \WJu \mid v \sile x \bigr\}$, 
and hence $v_{1} \le v_{x}$ by the maximality of $v_{x}$. 
Since $x \in ( (\WJu)_{\af} )_{\sige v_{x}}$ by the definition of $v_{x}$ and 
$v_{x} \ge v_{1}$, and since $x \in \kq{v_{1}}$, we see by the definition of $\kq{v_{1}}$ 
that $v_{1} = v_{x}$. Similarly, we see that $v_{2}=v_{x}$, 
and hence $v_{1}=v_{2}$, which is a contradiction. 
This proves the lemma. 
\end{proof}

We fix $\lambda \in P^{+} \subset P_{\af}^{0}$ 
(see \eqref{eq:P-fin} and \eqref{eq:P}), 
and take $\J=\J_{\lambda}$ as in \eqref{eq:J}. 
By Lemma~\ref{lem:dj2} and \eqref{eq:BK2}, we have
%
%
\begin{equation} \label{eq:dj3}
\gch V_{w}^{-}(\lambda) 
= \sum_{v \in \WJu_{\ge w}} \gch \Kg_{v}^{-}(\lambda)
\quad \text{for each $w \in \WJu$}.
\end{equation}

For each $i \in I$, 
we define a $\BC(q)$-linear operator $\sD_{i}$ on $\BC(q)[P]$ by
\begin{equation*}
\sD_{i} \be^{\nu} := 
\frac{ \be^{\nu - \rho} - \be^{s_{i}(\nu - \rho)} }{ 1-\be^{\alpha_{i}} } \be^{\rho} = 
\frac{ \be^{\nu}-\be^{\alpha_{i}}\be^{s_{i}\nu} }{ 1-\be^{\alpha_{i}} } \qquad 
\text{for $\nu \in P$, where $q=\be^{\delta}$}; 
\end{equation*}
note that $\sD_{i}^{2}=\sD_{i}$, and that
\begin{equation} \label{eq:Di}
\sD_{i} \be^{\nu} = 
\begin{cases}
\be^{\nu}(1+\be^{\alpha_{i}}+\be^{2\alpha_{i}}+ \cdots +\be^{-\pair{\nu}{\alpha_{i}^{\vee}}\alpha_{i}}) 
  & \text{if $\pair{\nu}{\alpha_{i}^{\vee}} \le 0$}, \\[1.5mm]
0 & \text{if $\pair{\nu}{\alpha_{i}^{\vee}}=1$}, \\[1.5mm]
-\be^{\nu}(\be^{- \alpha_{i}}+ \be^{- 2\alpha_{i}}+\cdots +\be^{(-\pair{\nu}{\alpha_{i}^{\vee}}+1)\alpha_{i}})
  & \text{if $\pair{\nu}{\alpha_{i}^{\vee}} \ge 2$}.
\end{cases}
\end{equation}
Also, we define a $\BQ(q)$-linear operator $\sT_{i}$ on $\BC(q)[P]$ by
\begin{equation*}
\sT_{i}:=\sD_{i}-1, \quad \text{that is}, \quad 
\sT_{i}\be^{\nu}=\frac{\be^{\alpha_i}(\be^{\nu}-\be^{s_i\nu})}{1-\be^{\alpha_i}};
\end{equation*}
note that $\sT_{i}^{2}=-\sT_{i}$ and $\sT_{i}\sD_{i}=\sD_{i}\sT_{i}=0$.
%
%
\begin{prop}[\cite{NS16}; see also {\cite[Proposition~6.6 and Remark~6.7]{NOS}}] \label{prop:dem}
Let $w \in \WJu$ and $i \in I$. Then, 
%
%
\begin{equation} \label{eq:dem1}
\sT_{i} \gch V_{w}^{-}(\lambda) = 
\begin{cases}
\gch V_{s_{i}w}^{-}(\lambda) - \gch V_{w}^{-}(\lambda)
 & \text{\rm if $\pair{w\lambda}{\alpha_{i}^{\vee}} < 0$}, \\[1mm]
0 & \text{\rm if $\pair{w\lambda}{\alpha_{i}^{\vee}} \ge 0$}.
\end{cases}
\end{equation}
\end{prop}

%
\begin{prop} \label{prop:rec1}
Let $w \in \WJu$ and $i \in I$. Then, 
%
%
\begin{equation} \label{eq:rec1}
\sT_{i} \gch \Kg_{w}^{-}(\lambda) = 
\begin{cases}
\gch \Kg_{s_{i}w}^{-}(\lambda)
 & \text{\rm if $\pair{w\lambda}{\alpha_{i}^{\vee}} < 0$}, \\[1mm]
- \gch \Kg_{w}^{-}(\lambda)
 & \text{\rm if $\pair{w\lambda}{\alpha_{i}^{\vee}} > 0$}, \\[1mm]
0 & \text{\rm if $\pair{w\lambda}{\alpha_{i}^{\vee}} = 0$}.
\end{cases}
\end{equation}
\end{prop}

\begin{proof}
We show \eqref{eq:rec1} by descending induction on $w$. 
Assume first that $w = \mcr{\lng}$; note that $\pair{\mcr{\lng}\lambda}{\alpha_{i}^{\vee}} = 
\pair{\lng\lambda}{\alpha_{i}^{\vee}} \le 0$. Also, by \eqref{eq:dj3}, we have 
$\gch V_{\mcr{\lng}}^{-}(\lambda) = \gch \Kg_{\mcr{\lng}}^{-}(\lambda)$. 
If $\pair{\mcr{\lng}\lambda}{\alpha_{i}^{\vee}} < 0$, then it follows from \eqref{eq:dem1} that
\begin{align*}
\sT_{i}\gch \Kg_{\mcr{\lng}}^{-}(\lambda) & = 
\sT_{i}\gch V_{\mcr{\lng}}^{-}(\lambda) = 
\gch V_{s_{i}\mcr{\lng}}^{-}(\lambda) - \gch V_{\mcr{\lng}}^{-}(\lambda) \\
& = \gch \Kg_{s_{i}\mcr{\lng}}^{-}(\lambda), 
\end{align*}
where the last equality follows from \eqref{eq:dj3} and the fact that
$\WJu_{\ge s_{i}\mcr{\lng}} = \bigl\{ s_{i}\mcr{\lng},\,\mcr{\lng} \bigr\}$. 
If $\pair{\mcr{\lng}\lambda}{\alpha_{i}^{\vee}} = 0$, 
then it follows from \eqref{eq:dem1} that
\begin{equation*}
\sT_{i}\gch \Kg_{\mcr{\lng}}^{-}(\lambda) = 
\sT_{i}\gch V_{\mcr{\lng}}^{-}(\lambda) = 0.
\end{equation*}
This proves \eqref{eq:rec1} in the case that $w = \mcr{\lng}$. 

Assume next that $w < \mcr{\lng}$. We set 
\begin{equation*}
(\WJu_{\ge w})_{0}:=
 \bigl\{ v \in \WJu_{\ge w} \mid \pair{v\lambda}{\alpha_{i}^{\vee}}=0 \bigr\}, \quad 
(\WJu_{\ge w})_{+}:=
 \bigl\{ v \in \WJu_{\ge w} \mid \pair{v\lambda}{\alpha_{i}^{\vee}}>0 \bigr\}, 
\end{equation*}
\begin{equation*}
(\WJu_{\ge w})_{-}':=
 \bigl\{ v \in \WJu_{\ge w} \mid 
 \text{$\pair{v\lambda}{\alpha_{i}^{\vee}} < 0$ and $s_{i}v \in \WJu_{\ge w}$} \bigr\}, 
\end{equation*}
\begin{equation*}
(\WJu_{\ge w})_{-}'':=
 \bigl\{ v \in \WJu_{\ge w} \mid 
 \text{$\pair{v\lambda}{\alpha_{i}^{\vee}} < 0$ and $s_{i}v \not\in \WJu_{\ge w}$} \bigr\}; 
\end{equation*}
observe that 
\begin{equation*}
\WJu_{\ge w} = 
(\WJu_{\ge w})_{0} \sqcup (\WJu_{\ge w})_{+} \sqcup 
(\WJu_{\ge w})_{-}' \sqcup (\WJu_{\ge w})_{-}'', \qquad 
(\WJu_{\ge w})_{-}' = s_{i}(\WJu_{\ge w})_{+}.
\end{equation*}

\paragraph{Case 1.}
%
Assume that $\pair{w\lambda}{\alpha_{i}^{\vee}} < 0$; note that 
$w \in (\WJu_{\ge w})_{-}''$. By \eqref{eq:dj3}, we have 
\begin{equation} \label{eq:rec1g}
\sT_{i} \gch V_{w}^{-}(\lambda)  = 
\sum_{v \in \WJu_{\ge w}} \sT_{i} \gch \Kg_{v}^{-}(\lambda). 
\end{equation}
The right-hand side of \eqref{eq:rec1g} is equal to: 
\begin{align*}
& 
\sum_{v \in (\WJu_{\ge w})_{0}} 
 \sT_{i} \gch \Kg_{v}^{-}(\lambda) + 
\sum_{v \in (\WJu_{\ge w})_{+}} 
 \sT_{i} \gch \Kg_{v}^{-}(\lambda) \\
& \hspace*{50mm} + 
\sum_{v \in (\WJu_{\ge w})_{-}'} 
 \sT_{i} \gch \Kg_{v}^{-}(\lambda) + 
\sum_{v \in (\WJu_{\ge w})_{-}''} 
 \sT_{i} \gch \Kg_{v}^{-}(\lambda). 
\end{align*}
Here, by the induction hypothesis, we see that
\begin{align*}
& \sum_{v \in (\WJu_{\ge w})_{0}} 
  \sT_{i} \gch \Kg_{v}^{-}(\lambda) = 0, \qquad 
  \sum_{v \in (\WJu_{\ge w})_{+}} 
  \sT_{i} \gch \Kg_{v}^{-}(\lambda) = 
  -\sum_{v \in (\WJu_{\ge w})_{+}} \gch \Kg_{v}^{-}(\lambda), \\
& \sum_{v \in (\WJu_{\ge w})_{-}'} 
  \sT_{i} \gch \Kg_{v}^{-}(\lambda) = 
  \sum_{v \in (\WJu_{\ge w})_{-}'} \gch \Kg_{s_{i}v}^{-}(\lambda) = 
  \sum_{v \in (\WJu_{\ge w})_{+}} \gch \Kg_{v}^{-}(\lambda). 
\end{align*}
Therefore, the right-hand side of \eqref{eq:rec1g} is equal to 
$\sum_{v \in (\WJu_{\ge w})_{-}''} 
 \sT_{i} \gch \Kg_{v}^{-}(\lambda)$. 
Hence, again by the induction hypothesis, we deduce that
\begin{equation} \label{eq:rec1b}
\sT_{i} \gch V_{w}^{-}(\lambda) 
 = \sT_{i} \gch \Kg_{w}^{-}(\lambda) + 
 \sum_{v \in (\WJu_{\ge w})_{-}'',\,v \ne w} 
 \gch \Kg_{s_{i}v}^{-}(\lambda).
\end{equation}
Also, by \eqref{eq:dem1} and \eqref{eq:dj3}, we have
\begin{equation} \label{eq:rec1c}
\sT_{i} \gch V_{w}^{-}(\lambda) = 
\gch V_{s_{i}w}^{-}(\lambda) - \gch V_{w}^{-}(\lambda) = 
\sum_{v \in \WJu_{\ge s_{i}w} \setminus \WJu_{\ge w}} 
\gch \Kg_{v}^{-}(\lambda).
\end{equation}
We claim that 
$\WJu_{\ge s_{i}w} \setminus \WJu_{\ge w} = s_{i}(\WJu_{\ge w})_{-}''$. 
Let $v \in \WJu_{\ge s_{i}w} \setminus \WJu_{\ge w}$. 
If $\pair{v\lambda}{\alpha_{i}^{\vee}} \le 0$, the 
we deduce from Lemma~\ref{lem:dia-b}\,(1), 
together with the assumption
$\pair{s_{i}w\lambda}{\alpha_{i}^{\vee}} > 0$ 
of Case 1, that $v \in \WJu_{\ge w}$, 
which is a contradiction. 
As a result, we have $\pair{v\lambda}{\alpha_{i}^{\vee}} > 0$, 
and hence $\pair{s_{i}v\lambda}{\alpha_{i}^{\vee}} < 0$. 
Since $\pair{s_{i}w\lambda}{\alpha_{i}^{\vee}} > 0$, 
it follows from Lemma~\ref{lem:dia-b}\,(3) that 
$s_{i}v \in \WJu_{\ge w}$. Since $s_{i}(s_{i}v) = v \notin \WJu_{\ge w}$, 
we obtain $s_{i}v \in (\WJu_{\ge w})_{-}''$, and hence 
$v \in s_{i}(\WJu_{\ge w})_{-}''$. Thus we have shown 
$\WJu_{\ge s_{i}w} \setminus \WJu_{\ge w} \subset s_{i}(\WJu_{\ge w})_{-}''$. 
In order to show the opposite inclusion, 
let $v \in (\WJu_{\ge w})_{-}''$. 
Then we see from Lemma~\ref{lem:dia-b}\,(3) that 
$s_{i}v \in \WJu_{\ge s_{i}w}$. In addition, by the definition of 
$(\WJu_{\ge w})_{-}''$, we have $s_{i}v \notin \WJu_{\ge w}$. 
Thus we obtain $s_{i}v \in \WJu_{\ge s_{i}w} \setminus \WJu_{\ge w}$, as desired. 

By the equality $\WJu_{\ge s_{i}w} \setminus \WJu_{\ge w} = s_{i}(\WJu_{\ge w})_{-}''$ 
and \eqref{eq:rec1c}, we see that
\begin{align*}
\sT_{i} \gch V_{w}^{-}(\lambda) & = 
\sum_{v \in s_{i}(\WJu_{\ge w})_{-}''} \gch \Kg_{v}^{-}(\lambda)=
\sum_{v \in (\WJu_{\ge w})_{-}''} \gch \Kg_{s_{i}v}^{-}(\lambda) \\
& = \gch \Kg_{s_{i}w}^{-}(\lambda) + 
\sum_{v \in (\WJu_{\ge w})_{-}'',\,v \ne w} 
\gch \Kg_{s_{i}v}^{-}(\lambda). 
\end{align*}
Combining this equality and \eqref{eq:rec1b}, we obtain 
$\sT_{i} \gch \Kg_{w}^{-}(\lambda) = \gch \Kg_{s_{i}w}^{-}(\lambda)$, as desired. 

\paragraph{Case 2.}
%
Assume that $\pair{w\lambda}{\alpha_{i}^{\vee}} > 0$; 
note that $w \in (\WJu_{\ge w})_{+}$. 
We see by Lemma~\ref{lem:dia-b}\,(2) that 
if $v \in \WJu_{\ge w}$ satisfies 
$\pair{v\lambda}{\alpha_{i}^{\vee}} < 0$, then 
$s_{i}v \in \WJu_{\ge w}$. Hence it follows that 
$(\WJu_{\ge w})_{-}''=\emptyset$, 
$(\WJu_{\ge w})_{-}'$ is identical to 
$(\WJu_{\ge w})_{-}:=\bigl\{v \in \WJu_{\ge w} \mid 
\pair{v\lambda}{\alpha_{i}^{\vee}} < 0\bigr\}$, and 
$s_{i}(\WJu_{\ge w})_{-} = (\WJu_{\ge w})_{+}$. 
By \eqref{eq:dj3}, we have
\begin{equation} \label{eq:rec1d}
\sT_{i} \gch V_{w}^{-}(\lambda) = 
 \sum_{v \in (\WJu_{\ge w})_{0}} \sT_{i} \gch \Kg_{v}^{-}(\lambda) + 
 \sum_{v \in (\WJu_{\ge w})_{+}} \sT_{i} \gch \Kg_{v}^{-}(\lambda) + 
 \sum_{v \in (\WJu_{\ge w})_{-}} \sT_{i} \gch \Kg_{v}^{-}(\lambda).
\end{equation}
Here we remark that 
$\sT_{i} \gch V_{w}^{-}(\lambda) = 0$ by \eqref{eq:dem1}. 
Also, by the induction hypothesis, 
we have 
\begin{equation*}
\sum_{v \in (\WJu_{\ge w})_{0}} \sT_{i} \gch \Kg_{v}^{-}(\lambda) = 0,
\end{equation*}
\begin{align*}
\sum_{v \in (\WJu_{\ge w})_{+}} \sT_{i} \gch \Kg_{v}^{-}(\lambda) & = 
\sT_{i} \gch \Kg_{w}^{-}(\lambda) + 
\sum_{v \in (\WJu_{\ge w})_{+},\,v \ne w} \sT_{i} \gch \Kg_{v}^{-}(\lambda) \\ 
& = \sT_{i} \gch \Kg_{w}^{-}(\lambda) -
\sum_{v \in (\WJu_{\ge w})_{+},\,v \ne w} \gch \Kg_{v}^{-}(\lambda), 
\end{align*}
\begin{equation*}
\sum_{v \in (\WJu_{\ge w})_{-}} \sT_{i} \gch \Kg_{v}^{-}(\lambda) = 
\sum_{v \in (\WJu_{\ge w})_{-}} \gch \Kg_{s_{i}v}^{-}(\lambda) = 
\sum_{v \in (\WJu_{\ge w})_{+}} \gch \Kg_{v}^{-}(\lambda). 
\end{equation*}
Substituting these equalities into \eqref{eq:rec1d}, we obtain 
$\sT_{i} \gch \Kg_{w}^{-}(\lambda) = - \gch \Kg_{w}^{-}(\lambda)$, as desired. 

\paragraph{Case 3.}
%
Assume that $\pair{w\lambda}{\alpha_{i}^{\vee}} = 0$; 
note that $w \in (\WJu_{\ge w})_{0}$. 
By the same argument as in Case 2, 
we see that 
$(\WJu_{\ge w})_{-}''=\emptyset$, 
$(\WJu_{\ge w})_{-}' = (\WJu_{\ge w})_{-}$, and 
$s_{i}(\WJu_{\ge w})_{-} = (\WJu_{\ge w})_{+}$; 
hence the equality \eqref{eq:rec1d} holds also in this case. 
Here we remark that $\sT_{i} \gch V_{w}^{-}(\lambda) = 0$ by \eqref{eq:dem1}, 
and that 
\begin{equation*}
\sum_{v \in (\WJu_{\ge w})_{0}} \sT_{i} \gch \Kg_{v}^{-}(\lambda) =  
\sT_{i} \gch \Kg_{w}^{-}(\lambda), 
\end{equation*}
\begin{equation*}
\sum_{v \in (\WJu_{\ge w})_{+}} \sT_{i} \gch \Kg_{v}^{-}(\lambda) = 
- \sum_{v \in (\WJu_{\ge w})_{+}} \gch \Kg_{v}^{-}(\lambda),
\end{equation*}
\begin{equation*}
\sum_{v \in (\WJu_{\ge w})_{-}} \sT_{i} \gch \Kg_{v}^{-}(\lambda) = 
\sum_{v \in (\WJu_{\ge w})_{+}} \gch \Kg_{v}^{-}(\lambda)
\end{equation*}
by the induction hypothesis. 
Substituting these equalities into \eqref{eq:rec1d}, we conclude that 
$\sT_{i} \gch \Kg_{w}^{-}(\lambda) = 0$, as desired. 

This completes the proof of Proposition~\ref{prop:rec1}. 
\end{proof}

Recall the following theorem from 
\cite{CO} and \cite[Theorem~35]{NNS2}. 
%
%
\begin{thm} \label{thm:CO}
Let $w \in \WJu$ and $i \in I$ be such that 
$\pair{w\lambda}{\alpha_{i}^{\vee}} < 0$. 
\begin{enu}
\item If $-\xcr{w}^{-1}\alpha_{i}$ is not a simple root, then 
\begin{equation} \label{co1}
\sT_{i} E_{w\lambda}(q,\infty) = E_{s_{i}w\lambda}(q,\infty). 
\end{equation}

\item If $-\xcr{w}^{-1}\alpha_{i}$ is a simple root, then 
\begin{equation} \label{co2}
\sT_{i} E_{w\lambda}(q,\infty) = 
(1-q^{\pair{\lambda}{\xcr{w}^{-1}\alpha_{i}^{\vee}}})E_{s_{i}w\lambda}(q,\infty). 
\end{equation}
\end{enu}
\end{thm}

For $w \in \WJu$, let $F_{w}(q)$ denote 
the right-hand side of \eqref{eq:main1}, that is, 
%
%
\begin{equation} \label{eq:F}
F_{w}(q):=
 \underbrace{\prod_{i \in I} 
 \prod_{r=1}^{\pair{\lambda}{\alpha_{i}^{\vee}}-\eps_{i}}
 \frac{1}{1-q^{-r}}}_{=:c_{w}(q)} E_{w\lambda}(q,\infty),
\end{equation}
where $\eps_{i}=\eps_{i}(\xw)$ is as in  \eqref{eq:eps}. 
%
%
\begin{lem} \label{lem:F}
Let $w \in \WJu$ and $i \in I$ 
be such that $\pair{w\lambda}{\alpha_{i}^{\vee}} < 0$. Then, 
\begin{equation} \label{co3}
\sT_{i} F_{w}(q) = F_{s_{i}w}(q). 
\end{equation}
\end{lem}

\begin{proof}
Since $s_{i}w \in \WJu$ 
(see Lemma~\ref{lem:236} and Remark~\ref{rem:236}), 
we have $\xcr{s_{i}w} = s_{i}ww_{\circ}(S) = s_{i}\xw$, 
where $w_{\circ}(S) \in \WJ$ is the longest element of $\WJ$ 
(see Section~\ref{subsec:SiBG}). 

Assume first that $-\xcr{w}^{-1}\alpha_{i}$ is not a simple root. 
Then, for each $k \in I$, we have
%
%
\begin{equation} \label{co3-1}
(\eps_{k}(\xcr{w}) = 1 \iff) \quad
\xw s_{k} > \xw \iff \xcr{s_{i}w} s_{k} > \xcr{s_{i}w} \quad
(\iff \eps_{k}(\xcr{s_{i}w})=1). 
\end{equation}
Indeed, this is verified as follows. 
Assume that $\xw s_{k} > \xw$. 
Then, $\xw \alpha_{k} \in \Delta^{+}$. Also, since 
$\pair{\xw\lambda}{\alpha_{i}^{\vee}} = 
 \pair{w\lambda}{\alpha_{i}^{\vee}} < 0$, 
we have $\xw^{-1}\alpha_{i} \in \Delta^{-}$. 
It follows that $\xw \alpha_{k} \ne \alpha_{i}$, and hence 
$s_{i}\xw \alpha_{k} \in \Delta^{+}$, 
which implies that $\xcr{s_{i}w}s_{k} = s_{i}\xw s_{k} > s_{i} \xw$. 
Assume that $\xcr{s_{i}w}s_{k} > \xcr{s_{i}w}$; 
we have $s_{i} \xw \alpha_{k} = \xcr{s_{i}w}\alpha_{k} \in \Delta^{+}$. 
Suppose, for a contradiction, that 
$\xw \alpha_{k} \in \Delta^{-}$. Then we have 
$\xw \alpha_{k} = - \alpha_{i}$, and hence 
$-\xw^{-1}\alpha_{i} = \alpha_{k}$. 
However, this contradicts the assumption that 
$-\xw^{-1}\alpha_{i}$ is not a simple root. 
Thus we have shown \eqref{co3-1}. 
By \eqref{co3-1}, we deduce that $c_{w}(q) = c_{s_{i}w}(q)$.
Therefore, \eqref{co3} follows from \eqref{co1} in this case. 

Assume next that $-\xw^{-1}\alpha_{i}=\alpha_{m}$ for some $m \in I$. 
Then, by exactly the same argument as for \eqref{co3-1}, we have
\begin{equation} \label{co3-2}
\begin{cases}
\xw s_{k} > \xw \iff \xcr{s_{i}w} s_{k} > \xcr{s_{i}w} 
  \quad \text{for all $k \in I \setminus \{m\}$}, \\
\xw s_{m} < \xw \quad \text{and} \quad \xcr{s_{i}w} s_{m} > \xcr{s_{i}w}, 
\end{cases}
\end{equation}
and hence
\begin{equation} \label{co3-3}
\frac{1}{1-q^{ -\pair{\lambda}{\alpha_{m}^{\vee}} }} c_{s_{i}w}(q) =  c_{w}(q);
\end{equation}
notice that $-\pair{\lambda}{\alpha_{m}^{\vee}} = \pair{\lambda}{\xw^{-1}\alpha_{i}^{\vee}}$. 
Therefore, \eqref{co3} follows from \eqref{co2} in this case. 
This proves the lemma. 
\end{proof}

Now, we prove the formula \eqref{eq:main1}, that is, 
\begin{equation}
\gch \Kg_{w}^{-}(\lambda) = F_{w}(q) \quad 
\text{for $w \in \WJu$}
\end{equation}
by descending induction on $w$. 
Assume first that $w = \mcr{\lng}$. 
Then we see by \eqref{eq:dj3} that 
$\gch \Kg_{\mcr{\lng}}^{-}(\lambda) = 
 \gch V_{\mcr{\lng}}^{-}(\lambda)$. 
Since $\xcr{\mcr{\lng}} = \lng$, it follows that 
$\eps_{k}(\xcr{\mcr{\lng}}) = 0$ for all $k \in I$, 
and hence 
\begin{equation*}
F_{\mcr{\lng}}(q) = 
 \prod_{i \in I} 
 \prod_{r=1}^{\pair{\lambda}{\alpha_{i}^{\vee}}}
 \frac{1}{1-q^{-r}} E_{\lng\lambda}(q,\infty).
\end{equation*}
Therefore, we deduce that $\gch \Kg_{\mcr{\lng}}^{-}(\lambda) = F_{\mcr{\lng}}(q)$ 
by \cite[Theorem~5.1.2]{NNS1}. 
Assume next that $w < \mcr{\lng}$, and take $i \in I$ such that 
$\pair{w\lambda}{\alpha_{i}^{\vee}} > 0$; 
we see by Lemma~\ref{lem:236} and Remark~\ref{rem:236} that 
$s_{i}w \in \WJu$ and $s_{i}w > w$. 
Hence, by the induction hypothesis, we obtain 
$\gch \Kg_{s_{i}w}^{-}(\lambda) = F_{s_{i}w}(q)$. 
Here, by Proposition~\ref{prop:rec1}, we have 
$\gch \Kg_{w}^{-}(\lambda) = \sT_{i} \gch \Kg_{s_{i}w}^{-}(\lambda)$. 
Also, by Lemma~\ref{lem:F}, we have $F_{w}(q) = \sT_{i}F_{s_{i}w}(q)$. 
From these equalities, we conclude that
\begin{equation*}
\gch \Kg_{w}^{-}(\lambda) = \sT_{i} \gch \Kg_{s_{i}w}^{-}(\lambda) 
= \sT_{i}F_{s_{i}w}(q) = F_{w}(q), 
\end{equation*}
as desired. This completes the proof of the formula \eqref{eq:main1}. 

\begin{rem}
Recall from \eqref{eq:dj3} that 
$\gch V_{w}^{-}(\lambda) 
= \sum_{v \in \WJu_{\ge w}} \gch \Kg_{v}^{-}(\lambda)$ 
for each $w \in \WJu$. 
By the M\"obius inversion formula, together with 
\cite[Corollary~2.7.10]{BB}, we deduce that
\begin{equation*}
\gch \Kg_{w}^{-}(\lambda) = 
\sum_{ v \in \WJu_{\ge w},\,[w,v] \subset \WJu }
(-1)^{\ell(v)-\ell(w)} \gch V_{v}^{-}(\lambda) \quad 
\text{for each $w \in \WJu$}, 
\end{equation*}
where we set $[w,v]:=\bigl\{u \in W \mid w \le u \le v\bigr\}$. 
\end{rem}
%
%
\section{Relations satisfied by the cyclic vector $u_{w} \in \Kg_{w}^{-}(\lambda)$.}
\label{sec:apdx}

Let $\lambda \in P^{+}$, and take $\J=\J_{\lambda}$ as in \eqref{eq:J}. 
For each $w \in W^{\J}$, let $u_{w} \in \Kg_{w}^{-}(\lambda)$ denote
the image of the cyclic vector $v_{w} \in V_{w}^{-}(\lambda)$ 
under the canonical projection $V_{w}^{-}(\lambda) \twoheadrightarrow \Kg_{w}^{-}(\lambda)$. 
Recall from Section~\ref{subsec:liealg} that for $i \in I_{\af}$, 
$F_{i} \in U_{\q}(\Fg_{\af})$ denotes 
the Chevalley generator corresponding to $-\alpha_{i}$. 
%
%
\begin{lem}[cf. {\cite[\S3.1]{FKM}}] \label{lem:rel} \mbox{}
\begin{enu}
\item For each $i \in I$, it holds that
\begin{equation*}
\begin{cases}
F_{i}^{\pair{w\lambda}{\alpha_{i}^{\vee}}}u_{w} = 0 & 
 \text{\rm 
 if $\pair{w\lambda}{\alpha_{i}^{\vee}} > 0$ ($\Leftrightarrow$ 
    $w^{-1}\alpha_{i} \in \Delta^{+} \setminus \DeJ^{+}$)}, \\[1.5mm]
F_{i}u_{w} = 0 & 
 \text{\rm 
 if $\pair{w\lambda}{\alpha_{i}^{\vee}} \le 0$ ($\Leftrightarrow$ 
    $w^{-1}\alpha_{i} \in \Delta^{-} \cup \DeJ$)}.
\end{cases}
\end{equation*}

\item It holds that
\begin{equation*}
\begin{cases}
F_{0}^{-\pair{w\lambda}{\theta^{\vee}}+1}u_{w} = 0 & 
 \text{\rm 
 if $\pair{w\lambda}{\theta^{\vee}} < 0$ ($\Leftrightarrow$ 
    $w^{-1}\theta \in \Delta^{-} \setminus \DeJ^{-}$)}, \\[1.5mm]
F_{0}u_{w} = 0 & 
 \text{\rm 
 if $\pair{w\lambda}{\theta^{\vee}} \ge 0$ ($\Leftrightarrow$ 
    $w^{-1}\theta \in \Delta^{+} \cup \DeJ$)}.
\end{cases}
\end{equation*}

\item It holds that
$U_{\q}^{-}(\Fg_{\af})_{w\alpha+m\delta}\,u_{w}=\{0\}$ 
for all $\alpha \in \Delta^{+}$ and $m \in \BZ$.
\end{enu}
\end{lem}

\begin{proof}
The first equality of part (1) follows from the fact that 
$F_{i}^{(\pair{w\lambda}{\alpha_{i}^{\vee}})}v_{w}$ is identical to 
$v_{s_{i}w} \in V_{s_{i}w}^{-}(\lambda)$; notice that 
$s_{i}w \in \WJu$ and $s_{i}w > w$ (see \cite[Lemma~5.9]{LNSSS}). 
The second equality of part (1) and part (2) follow from 
the fact that $v_{w}$ is an extremal weight vector of 
weight $w\lambda$. Part (3) follows from \cite[Theorem~5.3]{Kas02}. 
\end{proof}
%
%
{\small
 \setlength{\baselineskip}{12pt}

}

\end{document}